\documentclass[reqno,11pt]{amsart}
\usepackage[utf8]{inputenc}
\usepackage{float}
\usepackage{subcaption}

\usepackage[sort, numbers]{natbib}
\usepackage{graphicx}
\usepackage{amsmath}
\usepackage{amsthm}
\usepackage{amssymb}
\usepackage{xcolor}
\usepackage{ulem}
\usepackage{hyperref}

\usepackage{mathtools}
\newcommand{\defeq}{\vcentcolon=}

\newcommand{\compconj}[1]{%
  \overline{#1}%
}

\newtheorem{theorem}{Theorem}
\newtheorem{proposition}{Proposition}
\newtheorem{lemma}{Lemma}
\newtheorem{corollary}{Corollary}
\newtheorem{remark}{Remark}
\newtheorem{definition}{Definition}

\newcommand{\N}{\mathbb{N}}

\newcommand{\R}{\mathbb{R}}
\newcommand{\C}{\mathbb{C}}
\newcommand{\E}{\mathbb{E}}
\newcommand{\prob}{\mathbb{P}}
\newcommand{\itr}[1]{{B}}
\newcommand{\bou}[1]{{\partial B}}
\newcommand{\tot}[1]{{\compconj{B}}}
\newcommand{\statespace}{A}
\newcommand{\acc}[1]{{\left\{#1\right\}}}
\newcommand{\RNum}[1]{\uppercase\expandafter{\romannumeral #1\relax}}
\newcommand{\Law}{\text{Law}}

\begin{document}

\title[Stochastic Schr\"odinger-Lohe model]{Stochastic Schrödinger-Lohe model}
\author[Reika Fukuizumi]{Reika Fukuizumi$^{\scriptsize 1}$}
\author[Leo Hahn]{Leo Hahn$^{\scriptsize 2}$} 

\date{\today}


\maketitle

\begin{center} \small
$^1$ Research Center for Pure and Applied Mathematics, \\
Graduate School of Information Sciences, Tohoku University,\\
Sendai 980-8579, Japan; \\
\email{fukuizumi@math.is.tohoku.ac.jp}
\end{center}

\begin{center} \small
$^2$ D\'{e}partement de Math\'{e}matiques et Applications, \\
\'{E}cole Normale Sup\'{e}rieure, PSL University, \\
75005 Paris, France; \\
\email{leo.hahn@ens.fr}
\end{center}

\begin{abstract}

The Schrödinger-Lohe model consists of wave functions interacting with each other, according to a system of Schr\"odinger equations with a specific coupling such that all wave functions evolve on the $L^2$  unit ball. This model has been extensively studied over the last decade and it was shown that under suitable assumptions on the initial state, if one waits long enough all the wave functions become arbitrarily close to
each other, which we call a synchronization.
In this paper, we consider a stochastic perturbation of the Schrödinger-Lohe model and show a weak version of synchronization for this perturbed model in the case of two oscillators.

\end{abstract}

\vspace{0.3cm}

\begin{center} \small
\end{center}

\tableofcontents

\section{Introduction}

Synchronization is the emergence of a collective behavior inside a group of independent agents. Examples include crowds of people clapping at the same time, the collective flashing of some species of fireflies and the simultaneous electric activity of pacemaker cells in our hearts. This kind of phenomena is captured by the Kuramoto model that has been extensively studied. In \cite{lohe09}, Lohe proposed a non-abelian generalization of that model. In \cite{lohe10}, Lohe further developed the ideas of \cite{lohe09} with the goal of using synchronization in a quantum setting as a way to possibly duplicate quantum information. 
\vspace{3mm}

A special case of the model proposed in \cite{lohe10}, the Schrödinger-Lohe model, has been studied recently (see \cite{antonelli17, choi14, choi16, huh18} and the references therein).  
For the emergent dynamics, the analysis of the pairwise correlation functions between wave oscillators is essential, because we can reduce the analysis of the $L^2$ distance between wave oscillators to the analysis of their correlation using $L^2$ norm conservation.  
It seems that the first 
study on the (complete, practical) synchronization was achieved in \cite{choi14} under somewhat restricted initial conditions. This restriction was relaxed by   
\cite{antonelli17} in a natural manner, 
and a more refined complete
synchronization in terms of initial condition was derived in \cite{huh18} with a generalization of the model in term of different couplings. As one of generalizations of the model, the authors in \cite{choi16} considered 
the Schrödinger-Lohe model adding a potential term. 
\vspace{3mm}

In this paper, we investigate the Schrödinger-Lohe model when it is perturbed by a multiplicative noise. The use of a multiplicative noise is natural since it allows us to have conservation of mass in our model. We show that in this setting as well, a weak version of synchronization, in the case of two oscillators. 
\vspace{3mm}

The result in \cite{choi16} mentioned above is of particular interest to us because it tells us that a synchronization occurs in some weak sense provided  
the maximum difference among the values of potentials is small enough compared to the positive coupling strength of wave oscillators. If we interpret the potential terms as a deterministic perturbation of multiplicative type, and the maximum difference of potentials as the strength of noise, we may prove   
the same kind of synchronization concerning the balance between the coupling strength and the noise strength. This will be done indeed by an application of the large deviation principle in this paper.
\vspace{3mm}

We are interested in the large time behavior of the correlation functions as in the deterministic case. In the case of two oscillators, there is only one correlation function and it is possible to study the correlation function using explicit computations; we will consider that case for the remainder of the introduction. It turns out that our complex-valued correlation function satisfies a stochastic differential equation which is degenerate, in the sense that the dimension of the noise is strictly less than the dimension of the space where it lives. We will see in more detail about this stochastic differential equation that if the initial data is in 
the interior of the complex closed unit  ball, the solution stays in the interior, and if the initial data is on the boundary of the ball, the solution stays on the boundary. We can thus restrict the stochastic differential equation either to the interior or the boundary, and we find that 
the diffusion restricted to the boundary is non degenerate and that there exists a unique invariant (thus ergodic) probability measure on the boundary. By this ergodicity of the invariant probability measure, we prove that the modulus of solutions which start in the interior of the ball converges to one almost surely,  and this implies that there is no invariant probability measure in the interior. Since Doeblin's condition holds for the non-degenerate diffusion on the boundary, the law of solution whose initial data is on the boundary  converges to the invariant measure exponentially in the total variation distance, thus it is recurrent.

These results imply a synchronization in the sense of distributions i.e. the $L^2$ distance 
of any two wave oscillators converges to $\delta_0$ in distribution.
\vspace{3mm}

Two important questions regarding the stochastic 
model have not been answered in this work and will be pursued in the future. The first one is a modeling question: What happens for a different choice of noise and which noise makes the most sense from a physical point of view? 
The second one is on a synchronization result for more than two wave oscillators.  

\section{Preliminaries and Main results}

In this section, we precisely explain the results in this paper.  
We first give some notation. We denote by $L^2(\R^d, \C)$ the Lebesgue space of complex valued, square-integrable functions, and the inner product in the complex Hilbert space $L^2(\mathbb{R}^d, \C)$ is
denoted by, 
\begin{equation*}
\langle \psi, \varphi \rangle = \int_{\R^d} \psi(x)\overline{\varphi}(x) dx, \quad \psi, \varphi \in  L^2(\R^d, \C).
\end{equation*}
The norm in $L^2(\R^d, \C)$ is denoted by $\|\cdot\|$.
We define for an integer $k \ge 0$ the space $H^k(\R^d, \C)$ to be the set of all functions on $\R^d$ whose derivative up to $k$-times exist in the   
weak sense and is in $L^2(\R^d, \C)$. 
For any metric spaces $X$ and $Y$ let $C(X, Y)$ be the space of continuous functions from $X$ to $Y$. Let $C_b(X, Y)$ be the set of bounded continuous functions, and $B_b(X,Y)$ denote the set of Borel bounded functions with the norm denoted by $\lVert \cdot \rVert_{L^{\infty}}$.

We denote $B = \{ z \in \mathbb{C} : \lvert z \rvert < 1\}$ the complex open unit ball. It follows that $\compconj{B} = \{ z \in \mathbb{C} : \lvert z \rvert \leq 1\}$ is the complex closed unit ball and that $\partial B = \{ z \in \mathbb{C} : \lvert z \rvert = 1\}$ is the complex unit circle.
\vspace{0.1 in}

We now turn to the Lohe model which is a quantum version of Kuramoto model.  
Let $U_j$ and $U_j^{\dagger}$ be a unitary $d \times d$ matrix and its Hermitian conjugate, respectively, and let $H_j$ 
be the $d \times d$ Hermitian matrix whose eigenvalues correspond to the natural frequencies of the
oscillator at node $j$. We denote by $N \in \N$ the number of oscillators. 
Then, the Lohe model for quantum synchronization may be written as follows
\begin{equation} \label{eq:Lohe} 
i \frac{d}{dt}{U}_j =H_j U_j +i\frac{\kappa}{N} \sum_{k=1}^N a_{kj} 
(U_k-U_j U_k^{\dagger} U_j), \quad j=1, 2, ..., N.
\end{equation}
The positive constant $\kappa$ describes the attractive coupling strength, and $\{a_{kj}\}_{k,j}$ is the connectivity matrix. 
If we write $U_j=e^{-i\theta_j}$ and $H_j=\Omega_j$ in the case of $d=1$, the following Kuramoto model is derived.
\begin{equation} \label{eq:kuramoto}
\frac{d}{dt}
\theta_j=\Omega_j + \frac{\kappa}{N} \sum_{k=1}^N a_{kj} \sin(\theta_k-\theta_j), \quad j=1, 2, ..., N.
\end{equation}

In the model (\ref{eq:Lohe}),  any column of $U_j$ may be regarded as the $d$-component complex state vector 
$|\psi_j\rangle$ as a quantum oscillator, and 
(\ref{eq:Lohe}) may be generalized, inserting the Plank constant $\hbar$, to the Schr\"{o}dinger representation: 
\begin{equation} \label{eq:Lohe2}
i\hbar \frac{d}{dt} | \psi_j \rangle  = H_j | \psi_j \rangle  +\frac{i \kappa_{\hbar} }{N} 
\sum_{k=1}^N a_{kj} \left[ |\psi_k \rangle  - |\psi_j \rangle  \langle \psi_k | \psi_j  \rangle \right], 
\end{equation}
with the coupling constant $\kappa_\hbar$. 
    
As discussed in Section 6 of \cite{lohe10}, the equations (\ref{eq:Lohe2}) can be applied to 
infinite-dimensional systems. In the coordinate representations, (\ref{eq:Lohe2}) become the following system 
of nonlinear coupled partial differential equations:  
\begin{equation}  \label{eq:Lohe3}
i\partial_t\psi_j = H_j \psi_j +\frac{i \kappa_{\hbar}}{N} 
\sum_{k=1}^N a_{kj} \left(\psi_k -{\langle \psi_j, \psi_k \rangle} \psi_j \right), 
\end{equation}
where $\psi_j(t,x)$ is the coordinate representation of $|\psi_j \rangle$. One example of the Hamiltonian $H_j$ 
is $H_j = \frac{p_j^2}{2m_j}$, where $m_j$ is the mass of the oscillator at the node $j$, and $p_j$ is the corresponding 
momentum operator. When the nodes are coupled through a quantum network, the operator may be considered 
acting in a common space for any $j$, thus we may instead consider Hamiltonians $H_j=\frac{p^2}{2m_i}$, where 
$p$ denotes the common momentum operator which is represented as 
$p \leftrightarrow -i \hbar \nabla$ in the coordinate representation.     
\vspace{3mm}

In this paper, setting $a_{k,j}=1$ for all $k,j$ for simplicity, we consider the following dimensionless form of (\ref{eq:Lohe3}) as in \cite{choi14}, 
what we call Schr\"odinger-Lohe model. 
More exactly, in the Schr\"odinger-Lohe model, we consider oscillators that are modeled by their wave functions $\psi_1, \ldots, \psi_N$. The wave functions are coupled by the system of equations:
\begin{equation} \label{eq:SL_N}
i\partial_t\psi_j = - \Delta \psi_j +\frac{iK}{N} 
\sum_{k=1}^N \left(\psi_k -{\langle \psi_j, \psi_k \rangle} \psi_j \right), 
\end{equation}
where the initial wave functions, which we will denote by 
$\psi_{j,0}$ for $j=1, 2,..., N$, are normalized, i.e.,
$$
\|\psi_{j,0}\|= 1, \quad j=1, 2, ..., N.  
$$
Here, for the sake of simplicity we have written the coupling constant by $K$ instead of $\kappa_\hbar$.  
\vspace{3mm}

As far as we know, the first synchronization result for  (\ref{eq:SL_N}) can be found in 
\cite{choi14}. The following proposition, taken from \cite{antonelli17}, which offers a somewhat refined analysis compared to \cite{choi14},  states roughly that the distance between any pair of wave functions decays exponentially in time provided that the initial wave functions are sufficiently close to each other. \\






\begin{proposition} \label{prop:cauchy}
Let $k \geq 0$ be a fixed integer and let $K > 0$ be fixed. \\
Let $(\psi_{1, 0}, ..., \psi_{N, 0}) \in  H^{k}(\mathbb{R}^d, \mathbb{C}) \oplus \cdots \oplus H^{k}(\mathbb{R}^d, \mathbb{C})$ such that
$$ \|\psi_{j,0}\|= 1, \quad j=1, 2, ..., N. $$
The system of partial differential equations (\ref{eq:SL_N}) 
with initial data $(\psi_{1, 0}, ..., \psi_{N, 0})$  
has a unique solution $(\psi_1, \ldots, \psi_N)$ 
in $C([0, +\infty), H^k(\mathbb{R}^d, \mathbb{C}) \oplus \cdots \oplus H^k(\mathbb{R}^d, \mathbb{C}))$.
Furthermore, the $L^2$-norm of $(\psi_1, \ldots, \psi_N)$ is preserved, i.e.,
$$
\|\psi_j(t)\|= 1, \quad j=1, 2, ..., N,
$$
for all $t \geq 0$. 
If 
$$
\sum_{k = 1}^N \mathrm{Re} \langle \psi_{j,0}, \psi_{k,0} \rangle > 0 \text{ for all } j = 1, 2, \ldots, N, 
$$
then there exists $T > 0$ and $C > 0$ such that
$$
\max_{1 \leq j, k \leq N} \lVert \psi_j(t) - \psi_k(t) \rVert \leq C e^{-\frac{K}{2} t}
$$
for all $t \geq T$.
\end{proposition}







\vspace{3mm}

Effects by stochastic perturbations in the Kuramoto model 
have been a target of general interest and studied 
in, for ex., \cite{gp,k,tt}. The equation (\ref{eq:kuramoto}) with the simplest stochastic perturbation, is the one with an additive 
white noise, where the sensibility function is uniformly constant, $\varepsilon$: 
\begin{equation} \label{eq:stochastic_lohe_matrix_model}
\dot{\theta}_j=\Omega_j + \frac{K}{N} \sum_{k=1}^N \sin(\theta_k-\theta_j)+\varepsilon \xi_j(t), \quad j=1, 2, ..., N,
\end{equation}
where $\xi_j(t)$ is a Gaussian noise satisfying 
$$\langle \xi_j \rangle =0, \quad \langle \xi_j(t), \xi_k (s)\rangle =\delta_0(t-s) \delta_{jk}, \quad t \ge s\ge 0, ~j,k=1, 2,... ,N.$$  
Here, $\delta_0$ is the Dirac measure at the origin, and $\delta_{jk}$ is the Kronecker delta, i.e.
$\delta_{jk}=1$ if $j=k$ and $\delta_{jk}=0$ if $j \ne k$.  

If we again operate the same transform as above, i.e., 
$U_j=e^{-i\theta_j}$ and $H_j=\Omega_j$, we obtain 
\begin{equation*}
i \dot{U}_j =H_j U_j +i\frac{K}{N} \sum_{k=1}^N 
(U_k-U_j U_k^{\dagger} U_j)+\varepsilon U_j \xi_j(t), \quad j=1, 2, ..., N,  
\end{equation*}
which would be a natural extension as the Lohe model with a stochastic perturbation. 
\vspace{0.1 in}

In this paper, as we have already mentioned, we study influences 
of a noise on the synchronization for the Schr\"{o}dinger-Lohe model. 
As the first step, we consider a Stratonovich  multiplicative white noise in time, since the noise here is considered as the limit of processes with nonzero correlation length, and to have conservation of mass like in the deterministic case. Conservation of mass is a reasonable property to ask for from a physical perspective because the $\psi_k$ are wave functions so the square of their modulus can be interpreted as a probability density.
\vspace{3mm}

After having completed this work, we were told about the existence of the paper \cite{kim20}, where 
starting from an equation similar to (\ref{eq:stochastic_lohe_matrix_model}), 
the authors derive a matrix Lohe model under a stochastic perturbation and study its mean-field limit. 
We are interested in a different setting, namely infinite-dimensional function spaces and not matrices, and we will take no mean-field limit.
\vspace{3mm}

We now introduce the system of equations that we will study in the remainder of this paper. Let $(\Omega, \mathcal{F}, \prob)$ be a probability space endowed with a standard complete
filtration $\{\mathcal{F}_t\}_{t\ge 0}$. Let  
$\{\beta_j\}_{1 \le j \le N}$ be a family of independent one dimensional Brownian motions associated to $\{\mathcal{F}_t\}_{t\ge 0}$. We set 
$\xi_j(t)=\frac{d\beta_j(t)}{dt}$ for $j=1,2,..., N$, and we consider the following stochastic Schr\"odinger equation: 

\begin{equation} \label{eq:SSL_N} 
d \psi_j = i \Delta \psi_j dt + \frac{K}{N} \left( \sum_{k = 1}^N \psi_k - \langle \psi_j, \psi_k \rangle \psi_j \right) dt + i \varepsilon \psi_j \circ d\beta_j, \quad j=1, 2, ..., N, 
\end{equation}
where $K, \varepsilon > 0$, and $\circ$ denotes the Stratonovich product.  The first result of this paper is the existence of a solution of  (\ref{eq:SSL_N}): 
\begin{theorem} \label{thm:cauchy}
Let $K > 0$ and $\varepsilon > 0$. Let $\psi_{j,0} 
\in L^2(\R^d, \mathbb{C})$ satisfying $\|\psi_{j,0}\|=1,$ for $j=1, 2, ..., N.$
Then there exists a unique, $\{\mathcal{F}_t\}_{t\ge 0}$-adapted solution of the system (\ref{eq:SSL_N}), 
$(\psi_1, \ldots, \psi_N) \in C([0, \infty), L^2(\R^d, \mathbb{C}) \oplus \cdots \oplus L^2(\R^d, \mathbb{C})
)$ a.s. 
with $\psi_j(0)=\psi_{j,0}$. 
Moreover, the $L^2$ norm is conserved, i.e., 
$$\|\psi_j(t)\|=1, \quad \mbox{for all}~ t \ge 0,~ \prob \mbox{-a.s.},~ j=1, 2, ..., N. $$
\end{theorem}

A proof of Theorem 1 will be given in the next section for the sake of completeness, but the method fully follows \cite{dbf}, namely, by the use of gauge transform:
$$ \phi_j(t)=e^{-i\varepsilon \beta_j(t)} \psi_j(t), \quad j=1, 2, ..., N, $$
the equation ($\ref{eq:SSL_N}$) comes down to the following random partial differential equation: 
\begin{equation} \label{eq:DSL_N}
\partial_t \phi_j = i \Delta \phi_j +\frac{K}{N} 
\sum_{k=1}^N 
\left(e^{-i\varepsilon(\beta_j-\beta_k)}\phi_k -e^{i\varepsilon(\beta_j-\beta_k)} \langle \phi_j, \phi_k\rangle \phi_j \right), \quad j=1, 2, ..., N, 
\end{equation} 
which may be solved using the classical deterministic arguments pathwise. 
\vspace{3mm} 
 
We wish to show synchronization in some sense for this stochastic model. 
In this paper we consider the case $N = 2$ where there are only two wave functions. 
In the case of $N=2$ for the deterministic model (\ref{eq:SL_N}),  
Propositions 5.2 and 5.3 of \cite{antonelli17} say that  

\begin{proposition}
Let $(\psi_1, \psi_2) \in C([0,\infty), L^2(\R^d, \C) \oplus 
L^2(\R^d, \C))$ be the solution of $(\ref{eq:SL_N})$ with initial data $(\psi_{1,0}, \psi_{2,0}) 
\in L^2(\R^d, \C) \oplus 
L^2(\R^d, \C)$ such that
$\| \psi_{1, 0} \|= \| \psi_{2, 0} \|= 1$ obtained in Proposition \ref{prop:cauchy}. 
If
$$
\langle \psi_{1,0}, \psi_{2,0} \rangle \neq -1,
$$
then there exist $T > 0$ and $C > 0$ such that
$$
\| \psi_1(t) - \psi_2(t) \| \leq Ce^{-\frac{K}{2} t}
$$
for all $t \geq T$.
\end{proposition}

To obtain this synchronization result, the key quantity is the correlation function $h(t) = \langle \psi_1(t), \psi_2(t) \rangle$, since, 
by the $L^2$-norm conservation, 
\begin{eqnarray*}
\| \psi_1(t) - \psi_2(t) \|^2 
&=& \| \psi_1(t) \|^2 - 2 \mathrm{Re} \langle \psi_1(t), \psi_2(t) \rangle + \| \psi_2(t) \| ^2 \\
&=& 2 - 2 \mathrm{Re} \langle \psi_1(t), \psi_2(t) \rangle.
\end{eqnarray*}
It follows from the equations for $\psi_{j}(t)$ ($j=1,2$) that the correlation function $h(t)$ verifies
$$
\frac{d}{dt} h = K(1 - h^2),
$$
with $|h(0)| \le 1$. 
It may be seen that this ordinary differential equation has two stationary points $\pm 1$; $h=-1$ is unstable, $h=1$ is stable. More precisely, we can solve this ODE for $h$ to obtain, for all $t \ge 0$, 
$$
h(t) = 1 - \frac{1}{\left( \frac{1}{1 - h(0)} - \frac{1}{2}\right)e^{Kt} + \frac{1}{2}}
$$
or equivalently
$$
\langle \psi_1(t), \psi_2(t) \rangle = 1 - \frac{1}{\left( \frac{1}{1 - \langle \psi_{1, 0}, \psi_{2, 0} \rangle} - \frac{1}{2}\right)e^{Kt} + \frac{1}{2}},
$$
from which the synchronization result follows.
\vspace{3mm}


Thus, our main interest in this paper is in the behaviour of   
$h(t) = \langle \psi_1(t), \psi_2(t) \rangle$, where 
$(\psi_1(t), \psi_2(t))$ is the solution of the system 
(\ref{eq:SSL_N}) with $N=2$. It turns out that 
the stochastic process $(h(t))_{t \geq 0}$ 
satisfies the following stochastic differential equation:
\begin{equation} \label{eq:h}
d h = K(1 - h^2) dt + \sqrt{2} i \varepsilon h \circ dW,
\end{equation}
where $W = \frac{1}{\sqrt{2}}(\beta_1 - \beta_2)$ is a one dimensional Brownian motion associated to 
$\{\mathcal{F}_t\}_{t\ge 0}$.
\vspace{3mm}

This stochastic differential equation (\ref{eq:h}) is 
degenerate in the sense that $h(t)$ is two dimensional (complex valued) but the noise $dW$ 
is one dimensional. Note that we are interested only in the initial data $h(0)$ 
such that $|h(0)| \le 1$ because of the normalized condition $\|\psi_{1,0}\|=\|\psi_{2,0}\|=1$. 
We will see that in the case of $|h(0)| \le 1$, the state space for 
the stochastic process $(h(t))_{t\ge 0}$ is $\compconj{B}$ 
and moreover that there are two invariant sets, $B$ and $\partial B$: if $|h(0)|<1$, then $|h(t)|<1$ for all $t>0$, and if $|h(0)|=1$, then $|h(t)|=1$ for all $t>0$.  
\vspace{3mm}

In order to know the time asymptotic behaviour of $h(t)$, 
it is natural to raise up questions on the existence of an invariant measure, the uniqueness of invariant measures, and 
the convergence to the invariant measure. 
\vspace{3mm}

Let $(P_t^\statespace)_{t \geq 0}$ be the Markov semigroup associated with (\ref{eq:h}), namely, 
$$
P_t^{\statespace} f(x) = \mathbb{E}(f(h(t))), \quad t\ge 0
$$
for $f \in B_b(\statespace,\R)$. We will take $A$ among the state spaces $B, \partial B, \overline{B}$ depending on the situation. $\delta_x P_t^{\statespace}$ 
means the law of solution of (\ref{eq:h}) starting from $x \in A$ (for the definition, see Section 4.1).

We remark that the diffusion on $\partial B$ defined by Eq.(\ref{eq:h}) is non-degenerate, thus many well-known analysis are available, and since $\partial B$ is compact, there exists a unique invariant measure denoted by $\mu_{K, \varepsilon}$.  
On the other hand, the diffusion on $B$ is degenerate.

Since $\mu_{K, \varepsilon}$ is the unique invariant measure, 
it is ergodic. Using this ergodicity, 
we have the following convergence result  
for the solution $h(t)$ starting 
at the interior of the unit ball $B$.  

\begin{theorem} \label{thm:conv_interieur}
The solution $h(t)$ of 
the stochastic differential equation (\ref{eq:h})
with $h(0) \sim \delta_x$ for a fixed $x \in B$ satisfies 
$$\lim_{t \rightarrow +\infty} \lvert h(t) \rvert = 1$$
almost surely.
\end{theorem}

Thanks to this convergence in $B$, we conclude the following Theorem.

\begin{theorem} \label{thm:mesure_inv}
Fix $K > 0$ and $\varepsilon > 0$ and consider the stochastic differential equation (\ref{eq:h}) and associated Markov semigroup $(P_t^\statespace)_{t\ge 0}$.
\begin{itemize}
\item[(i)] On $\partial B$ there exists a unique invariant probability measure $\mu_{K, \varepsilon}^{\partial B}$ 
with respect to $(P_t^{\partial B})_{t\ge 0}$.
\item[(ii)] On $B$ there is no invariant probability measure. 
\item[(iii)] On $\compconj{B}$ 
there exists a unique invariant probability measure $\mu_{K, \varepsilon}^{\compconj{B}}$ 
with respect to $(P_t^{\compconj{B}})_{t\ge 0}$, where
$\mu_{K, \varepsilon}^{\compconj{B}}$ is the natural extension of $\mu_{K, \varepsilon}^{\partial B}$ to $\compconj{B}$.
\end{itemize}
\end{theorem}

In a slight abuse of notation, we will denote $\mu_{K, \varepsilon}$ the unique invariant measure $\mu_{K, \varepsilon}^{\partial B}$ on $\partial B$ and also its natural extension to $\compconj{B}$.
\vspace{3mm}

Theorem \ref{thm:conv_interieur} is proved  showing that  
the solution of (\ref{eq:h}) starting in $B$ converges to the stationary solution having the law $\mu_{K, \varepsilon}$. Therefore, 

\begin{corollary} \label{cor:conv_weak} 
Fix $K > 0$ and $\varepsilon > 0$. We have, $$
\delta_x P_t^{B} \xrightarrow{t \rightarrow +\infty} \mu_{K, \varepsilon}
$$
in distribution for all $x \in B$. 
\end{corollary}

The compactness of $\partial B$ implies immediately the minorization condition (see for example \cite{msh}), which is a kind of Doeblin's 
condition, and if we are on $\partial B$ the meaning of the convergence of the law of solution is stronger, i.e. exponential convergence in the total variation distance. 

\begin{proposition} \label{prop:conv_bord} Fix $K > 0$ and $\varepsilon > 0$. 
Consider the Markov semigroup $(P_t^{\partial B})_{t \geq 0}$ associated with the stochastic differential equation 
(\ref{eq:h}) on $\partial B$. 
There exist $C, \gamma_* > 0$ such that
$$
\| \delta_x P_t^{\partial B} - \mu_{K, \varepsilon} \|_\text{TV} 
\leq C e^{- \gamma_* t}
$$
for every $x \in \partial B$ and for all $t\ge 1$.
\end{proposition}

\begin{figure}[h]

\begin{subfigure}{0.45\textwidth}
\includegraphics[width=1.1\linewidth, height=1.1\linewidth]{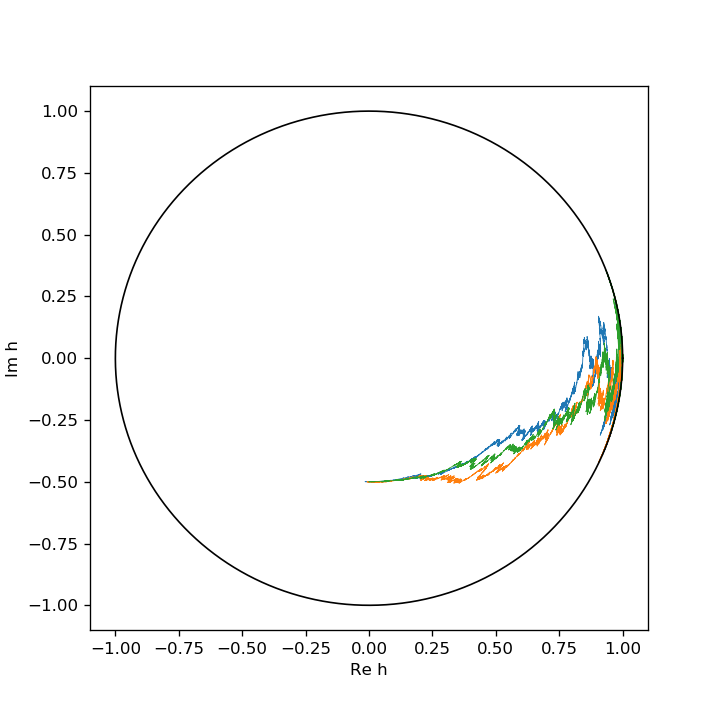} 
\caption{$K = 1$ and $\varepsilon = 0.2$}
\end{subfigure}
\begin{subfigure}{0.45\textwidth}
\includegraphics[width=1.1\linewidth, height=1.1\linewidth]{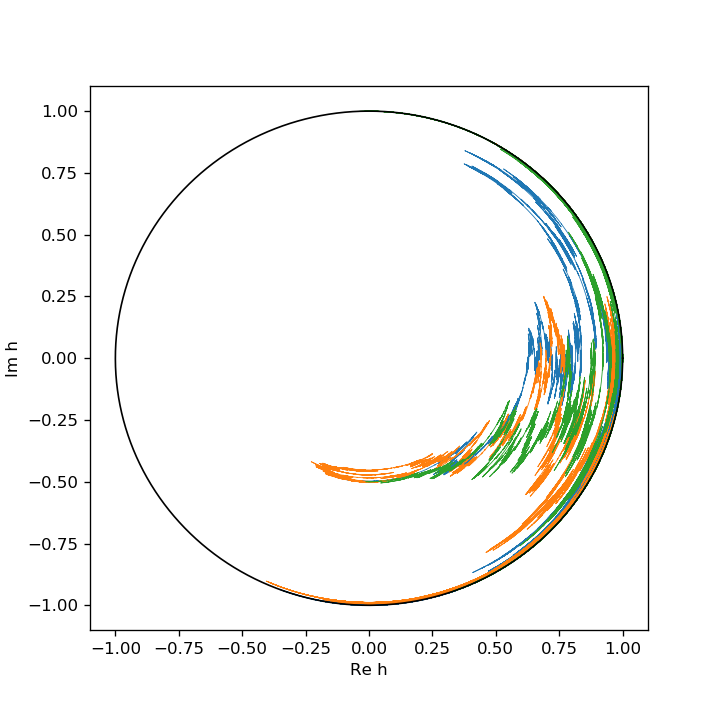}
\caption{$K = 1$ and $\varepsilon = 0.8$}
\end{subfigure}

\caption{Paths of $(h_t)_{t \geq 0}$ on the time interval $[0, 10]$}
\label{fig:image2}
\end{figure}

\begin{remark} \label{rem:1}
By Corollary \ref{cor:conv_weak} and the property of the density of $\mu_{K, \varepsilon}$ (see Proposition \ref{prop:esp_re} below), it may be also seen that 
$h(t)$ starting in the interior of the ball $B$ satisfies, after a long time passed,
$\E(\mathrm{Re} \, h(t))>0$. Namely, it will mostly be on the right side of $\bar{B}$ for the large time.
\end{remark}

\begin{remark}
By the convergence result Proposition \ref{prop:conv_bord} and 
Proposition \ref{prop:densite} below, applying 
Proposition 3.4.5 of \cite{dpz}, 
$h(t)$ is recurrent with respect to any small neighborhood $\mathcal{U} \subset \partial B$ of $1$. Thus, $h(t)$ with $h(0) \in \partial B$ leaves and returns to $\mathcal{U}$ infinitely many times. 
\end{remark}
\vspace{3mm}

All these convergence results in hand, we may have the following synchronization results for (\ref{eq:SSL_N}) with $N=2$. 

\begin{theorem} \label{thm:main}
Fix $K > 0$ and $\varepsilon > 0$. Let $(\phi_1, \phi_2)$ be a $\{\mathcal{F}_t\}_{t\ge 0}$-adapted process with paths in $C([0, \infty), L^2(\R^d, \C) \oplus L^2(\R^d, \C))$ satisfying
\begin{equation} \label{eq:SSL_2}
\left\{
\begin{array}{ll}
d \psi_1 &= i \Delta \psi_1 dt + \frac{K}{2} \left( \psi_2 - \langle \psi_1, \psi_2 \rangle \psi_1 \right) dt + i \varepsilon \psi_1 \circ d\beta_1\\
d \psi_2 &= i \Delta \psi_2 dt + \frac{K}{2} \left( \psi_1 - \langle \psi_2, \psi_1 \rangle \psi_2 \right) dt + i \varepsilon \psi_2 \circ d\beta_2
\end{array}
\right.
\end{equation}
with initial data $(\psi_1(0), \psi_2(0))=(\psi_{1,0}, \psi_{2,0}) \in L^2(\mathbb{R}^d, \mathbb{C}) \oplus
L^2(\mathbb{R}^d, \mathbb{C})$ such that 
$\lVert \psi_{1,0} \rVert = \lVert \psi_{2,0} \rVert = 1$ obtained in Theorem \ref{thm:cauchy}.
Then, 
\begin{itemize}
\item[(i)] There exists a probability measure $\nu_{K, \varepsilon}$ on $\mathbb{R}$ such that
$$
\mathrm{Law} (\lVert \psi_1(t) - \psi_2(t) \rVert) \xrightarrow{t \rightarrow +\infty} \nu_{K, \varepsilon}
$$
in distribution. This limit $\nu_{K, \varepsilon}$ depends only on $\frac{\varepsilon^2}{K}$ and not on $(\psi_{1, 0}, \psi_{2, 0})$.
\item[(ii)] Moreover, we have 
$$\lim_{t \to \infty} \inf_{\theta \in \R} \|\psi_1(t) - e^{i\theta}\psi_2(t)\| =0,$$ 
almost surely.
\end{itemize}
\end{theorem}
\vspace{0.1 in}

For $\varepsilon>0$ small, by the cerebrated Freidlin-Wenzell Theory \cite{fv}, one can see that the supremum of $|h(t)-1|$ on a finite time interval is small almost surely,  
but due to the noise effective to the angular direction, for any $\varepsilon>0$, $h(t)$ exits from any small neighborhood of $1$ in a finite time almost surely. If $\varepsilon>0$ is small, the probability of this exit should be small and the analysis of this rare event is essential to know the asymptotics of this probability involving the behavior of $h(t)$ for small $\varepsilon>0$ by the use of large deviation principle. In a similar spirit, but simply making use of the explicit formula of the density of $\mu_{K, \varepsilon}$ (see Proposition \ref{prop:densite} below)  
we derive the limit behavior as $\frac{\varepsilon^2}{K} \to 0$
of the invariant measure $\mu_{K, \varepsilon}$ which implies the following proposition in term of synchronization. Note that it is natural to consider 
the limit $\frac{\varepsilon^2}{K} \to 0$ by scaling of equation (\ref{eq:h}).  

\begin{proposition} \label{prop:conv_eps_to_zero}
Under the same assumption of Theorem \ref{thm:main},  
the probability measure $\nu_{K, \varepsilon}$ of (i) of Theorem \ref{thm:main} converges to $\delta_0$ when $\frac{\varepsilon^2}{K} \rightarrow 0$.
\end{proposition}

In other words, when $K > 0$ and $\varepsilon > 0$ are fixed and $t \rightarrow +\infty$ we have that $\lVert \psi_1(t) - \psi_2(t) \rVert$ converges to a limit that is very close to $\delta_0$ provided that $\frac{\varepsilon^2}{K}$ is close enough to $0$. This could be interpreted 
as a similar scenario 
of the result in \cite{choi16}, which considers the system of equations adding a potential term:
$$
i\partial_t\psi_j = - \Delta \psi_j + V_j \psi_j + \frac{iK}{N} 
\sum_{k=1}^N \left( \psi_k -{\langle \psi_j, \psi_k \rangle} \psi_j \right), \quad j=1, 2, ..., N,
$$
where $V_j = V_j(x, t)$ are $B_b(\R^d_x \times \R_t, \R)$ functions: 
The authors of \cite{choi16} actually showed 
that a synchronization occurs provided that the coupling strength $K$ is much larger than $\max_{jk} \lVert V_j - V_k \rVert_{L^\infty}$. 
\vspace{3mm} 

The paper is organized as follows. We show the global existence and uniqueness of solutions to (\ref{eq:SSL_N}) in Section 3. We investigate the existence, uniqueness, and the density of invariant measures in Section 4. The large deviation principle of the invariant measure on $\partial B$ is also given in Section~4. Section 5 is devoted to the convergence results of the law of solutions on $B$ and $\partial B$.

\section{Existence of solution in the stochastic case}

In this section, we establish the global existence, uniqueness and $L^2$-norm  conservation of solution for (\ref{eq:SSL_N}). As mentioned in the introduction, it suffices to prove the following Proposition for Theorem \ref{thm:cauchy} on the existence of the solution to (\ref{eq:DSL_N}). Let $k\ge 0$ be a fixed integer in what follows. 

\begin{proposition} Fix $K > 0$ and $\varepsilon > 0$. 
Let $\psi_{j,0} \in H^k(\R^d, \C)$ with $\|\psi_{j,0}\|=1$ for $j=1, 2, ..., N.$
There exist a unique, $\{\mathcal{F}_t\}_{t\ge 0}$-adapted solution of (\ref{eq:DSL_N}), $(\phi_1, \phi_2, ..., \phi_N) \in 
C([0,+\infty), H^k(\R^d, \C) \oplus \cdots \oplus H^k(\R^d, \C))$, a.s. 
with 
$(\phi_1(0), \phi_2(0), ..., \phi_N(0))=(\psi_{1,0}, \psi_{2,0},..., \psi_{N,0})$. Moreover, 
$$ \|\phi_{j}(t)\|=\|\psi_{j,0}\|=1, \quad j=1, 2, ...., N,$$
for all $t \geq 0$, almost surely.
\end{proposition} 

\begin{proof}
We write the equation (\ref{eq:DSL_N}) in the mild form:
\begin{eqnarray*}
[\mathcal{T}^{\omega} u](t)=U(t)u_0 + \int_0^t U(t-s) F_\omega(s, u(s)) ds,
\end{eqnarray*}
where we put $u=(\phi_1, \ldots, \phi_N)$,  
$u_0 = (\psi_{0, 1}, \ldots,\psi_{0, N})$, $(U(t))_{t\ge 0}$ is the  semigroup generated by the component-wise Laplacian and
\begin{eqnarray*}
F_{\omega}(t, u(t)) &=& F_\omega(t, (\phi_1(t),...,\phi_N(t))) \\
&=& \frac{K}{N}\Big(\sum_{k=1}^N e^{-i\varepsilon (\beta_1(t)-\beta_k(t))}\phi_k(t) 
- e^{i\varepsilon(\beta_1(t)-\beta_k(t))}\langle \phi_1(t), \phi_k(t) \rangle \phi_1(t), ....,\\
&& \hspace{3mm}\sum_{k=1}^N e^{-i\varepsilon (\beta_N(t)-\beta_k(t))}\phi_k(t) - e^{i\varepsilon(\beta_N(t)-\beta_k(t))}\langle \phi_N(t), \phi_k(t) \rangle \phi_N(t) \Big).
\end{eqnarray*}

We set $H = H^k(\mathbb{R}^d, \C) \oplus \cdots \oplus H^k(\mathbb{R}^d, \C)$, denote the norm by $\|\cdot \|_{H}$.  
We also set $X(T) = C([0, T]; H)$, and   
$$\bar{B}_T(0, R):=\{v \in X(T),~ \|v\|_{X(T)} \le R\},$$
with $R=2\|u_0\|_H.$
Since we have 
$$\|F_\omega(t, u)\|_{H} \le NK(\|u\|_H +\|u\|_H^3),$$
using the fact that $U(t)$ is a unitary operator in $H$ for all $t \geq 0$ 
we get for $t\in [0,T],$ if $u \in \bar{B}_T(0,R)$,
\begin{eqnarray*}
\Big\| [\mathcal{T}^{\omega} u](t) \Big\|_{H} 
&\leq& \| U(t)u_0 \|_{H} + \int_0^t \| U(t-s)F_\omega(s, u(s)) \|_{H} ds \\  
&\leq& \| u_0\|_{H} + T \underset{s \in [0, T]}{\sup} 
\| F_\omega(t, u(s))\|_{H} \\
&\leq& \frac{R}{2} + TNK(R+R^3).
\end{eqnarray*}
Moreover, 
since for $u,v \in H$, for all $t \geq 0$
$$
\|F_\omega(t, u) - F_\omega(t, v)\|_H \leq NK( \|u - v \|_H 
+ \Tilde{C} \| u - v \|_H 
(\|u \|_H^2 + \|v \|_H^2)), 
$$
if $u,v \in \bar{B}_T(0,R)$, then we have 
$$
\Big\lVert \mathcal{T}^{\omega}u - \mathcal{T}^{\omega}v \Big\rVert_{X(T)} \leq T NK(1 + 2\Tilde{C} R^2) \lVert u - v \rVert_{X(T)}.
$$
Thus, 
$\mathcal{T}^{\omega}$ is a contraction on $\bar{B}_T(0,R)$ since $U(t)$ is continuous in time, for $T$ sufficiently small. Note that $T$ depends only on $\|u_0\|_H$ and $\omega$.  
We can now use the Banach fixed-point theorem to obtain the existence of $u \in C([0, T], H)$, a solution of 
the mild form.
The standard argument in \cite{c} allows to have also   
the uniqueness and the blow-up alternative.  
By the $H^k$-regularity argument (see Chapter 5 of \cite{c}), there exists the maximal existence time $T^*=T^*(\|u_0\|_{L^2 \oplus \cdots \oplus L^2}, \omega)>0$, and 
if $T^*<+\infty$ then 
$\lim_{t \uparrow T^*}\|u(t)\|_{L^2 \oplus \cdots \oplus L^2} =+\infty.$

\vspace{3mm}

We prove the conservation of $L^2$ norm. This can be seen formally as follows. 
For a fixed $j\in \{1,..., N\},$
\begin{eqnarray*}
\partial_t \lVert \phi_j \rVert^2 
&=& 2 \mathrm{Re} \langle \partial_t \phi_j, \phi_j \rangle \\
&=& 2 \mathrm{Re} \langle i \Delta \phi_j 
+ \frac{K}{N}\sum_{k=1}^N (e^{-i\varepsilon(\beta_j-\beta_k)} \phi_k - 
e^{i\varepsilon(\beta_j-\beta_k)} 
\langle \phi_j, \phi_k \rangle \phi_j), \phi_j \rangle \\
&=& 2 \mathrm{Re} \langle i \Delta \phi_j, \phi_j \rangle 
 + \frac{2K}{N} \mathrm{Re}[\sum_{k=1}^N 
e^{-i\varepsilon(\beta_j-\beta_k)} \langle \phi_k, \phi_j \rangle] \\ 
&& - \frac{2K}{N} \mathrm{Re} [\sum_{k=1}^N e^{i\varepsilon(\beta_j-\beta_k)}
\langle \phi_j, \phi_k \rangle \langle \phi_j, \phi_j \rangle] \\
&=& \frac{2K}{N} \left\{\sum_{k=1}^N \mathrm{Re} [e^{-i\varepsilon(\beta_j-\beta_k)}\langle \phi_k, \phi_j \rangle]\right\} 
(1 - \lVert \phi_j \rVert^2).
\end{eqnarray*}
Hence 
$$\lVert \phi_j \rVert^2 = (\lVert \psi_{j, 0} \rVert^2 - 1) 
e^{-\frac{2K}{N} \int_0^t \sum_{k=1}^N \mathrm{Re} [e^{-i\varepsilon(\beta_j(s)-\beta_k(s))} \langle \phi_k(s), \phi_j(s)\rangle] ds}
+ 1.$$ 
This implies that $\lVert \phi_j(t) \rVert = 1$ for all $t \geq 0$ when $\lVert \psi_{j, 0}\rVert = 1$. 

When the initial wave functions $\psi_{j, 0}$ is in $H^2(\R^d, \mathbb{C})$ 
then $t \mapsto \lVert \phi_j \rVert^2$ is differentiable and we can actually do the formal computation above and obtain that $\lVert \phi_j(t) \rVert = 1$ 
for all $t \geq 0$. 

Now we assume that $\psi_{j,0} \in L^2(\R^d, \C)$ for $j=1, 2, ..., N$. 
In this case, we take $\left(\psi_{j, 0}^{(n)}\right)_{n \in \mathbb{N}}$ 
to be a sequence of functions in $H^2(\R^d, \mathbb{C})$ such that
$$
\lVert \psi_{j, 0}^{(n)} \rVert = 1 \text{ for all } n \in \mathbb{N}
$$
and
$$
\lim_{n \rightarrow +\infty} \psi_{j, 0}^{(n)} = \psi_{j, 0},
$$
where the convergence takes place in $L^2(\R^d, \mathbb{C})$. 
Let $(\phi_1^{(n)}, \phi_2^{(n)}, ..., \phi_N^{(n)})$ be the solution of $(\ref{eq:DSL_N})$ with initial data 
$(\psi_{1,0}^{(n)}, \psi_{2,0}^{(n)}, ..., \psi_{N,0}^{(n)})$. 
Then $\lVert \phi_j^{(n)}(t)\rVert = 1$ for all $t \geq 0$, all $j=1,2,..., N$, 
and all $n \in \mathbb{N}$. 
Furthermore, for $t \geq 0$ fixed, 
$$\lVert \phi_j^{(n)}(t) \rVert \xrightarrow{n \rightarrow +\infty} \lVert \phi_j(t) \rVert.$$ 
This implies that the $L^2$ norm is conserved, equals $1$, and further $T^*=+\infty$. Finally, $\{\mathcal{F}_t\}_{t\ge 0}$-adaptivity follows from that fact that $u(t)$ is obtained by a fixed point method, using the cut-off argument in $F_\omega(t, u)$. 
\end{proof}

From now on, we consider the case $N=2$, and we consider only two wave functions $(\psi_1, \psi_2)$ which are the solution of the system (\ref{eq:SSL_2}).  
\vspace{0.1 in}

Recall that our goal is to analyse the asymptotic behavior of 
$\| \psi_1 - \psi_2 \|$. Since we have already shown 
that the $L^2$-norm of $\psi_1$ and $\psi_2$ is preserved, 
the study of $\| \psi_1(t) - \psi_2(t) \|$ reduces to the study of $\langle \psi_1(t), \psi_2(t) \rangle$. We will thus define the stochastic process $(h(t))_{t \geq 0}$ defined by $h(t)= \langle \psi_1(t), \psi_2(t) \rangle$ and show that it is the solution of a stochastic differential equation. The study of this stochastic differential equation will allow us to conclude our analysis of the asymptotic behavior of $\| \psi_1 - \psi_2 \|$. 

\begin{lemma} Fix $K > 0$ and $\varepsilon > 0$. 
Let $(\psi_1, \psi_2) \in C([0,+\infty), L^2(\R^d, \C) \oplus L^2(\R^d, \C))$ 
be the solution of (\ref{eq:SSL_2}) 
with initial data $\psi_{1, 0}, \psi_{2, 0} \in L^2(\mathbb{R}^d, \mathbb{C})$ satisfying $\|\psi_{1, 0} \|= \| \psi_{2, 0} \|= 1.$ 
Then, the stochastic process $(h(t))_{t \geq 0}$ defined by $h(t) = \langle \psi_1(t), \psi_2(t) \rangle$ satisfies the following stochastic differential equation:
$$
d h= K(1 - h^2) dt + \sqrt{2} i \varepsilon h\circ dW,
$$
where $W = \frac{1}{\sqrt{2}}(\beta_1 - \beta_2)$ is a Brownian motion associated to the filtration $\{\mathcal{F}_t\}_{t\ge 0}$.
\end{lemma}

\begin{proof} 
Let $j=1,2.$ Take $\left(\psi_{j, 0}^{(n)}\right)_{n \in \mathbb{N}}$ 
to be a sequence of functions in $H^2(\R^d, \mathbb{C})$ such that
$\lVert \psi_{j, 0}^{(n)} \rVert = 1 \text{ for all } n \in \mathbb{N}$
and $\lim_{n \rightarrow +\infty} \psi_{j, 0}^{(n)} = \psi_{j, 0}$ 
in $L^2(\R^d, \mathbb{C})$. 
Let $(\phi_1^{(n)}, \phi_2^{(n)})$ be the solution of $(\ref{eq:DSL_N})$ with initial data $(\psi_{1,0}^{(n)}, \psi_{2,0}^{(n)})$. Then $\lVert \phi_j^{(n)}(t)\rVert = 1$ for all $t \geq 0$, $n\in \N$. 
Let $h_n(t)=e^{-\sqrt{2}i\varepsilon W(t)} \langle \phi_1^{(n)}, \phi_2^{(n)}\rangle$.  Note that 
$t \mapsto \langle \phi_1^{(n)}, \phi_2^{(n)}\rangle$  is differentiable, thus 
\begin{eqnarray*}
&&\partial_t \langle \phi_1^{(n)}, \phi_2^{(n)}\rangle \\
&=& \langle \partial_t \phi_1^{(n)}, \phi_2^{(n)}\rangle +
\langle \phi_1^{(n)}, \partial_t \phi_2^{(n)}\rangle \\
&=& 
\langle -i\left(-\Delta \phi_1^{(n)}  
+ \frac{iK}{2}(e^{\sqrt{2}i\varepsilon W} \phi_2^{(n)}
-e^{-\sqrt{2}i\varepsilon W}\langle \phi_1^{(n)}, \phi_2^{(n)} \rangle \phi_1^{(n)})\right), \phi_2^{(n)}\rangle \\
&&+
\langle \phi_1^{(n)}, -i\left(-\Delta \phi_2^{(n)} 
+ \frac{iK}{2}(e^{\sqrt{2}i\varepsilon W} \phi_1^{(n)}
-e^{-\sqrt{2}i\varepsilon W}
\langle \phi_2^{(n)}, \phi_1^{(n)} \rangle \phi_2^{(n)}\right)\rangle \\
&=& 
\frac{K}{2} e^{\sqrt{2}i\varepsilon W} (\|\phi_2^{(n)}\|^2+\|\phi_1^{(n)}\|^2) 
-K e^{-\sqrt{2}i\varepsilon W} (\langle \phi_1^{(n)}, \phi_2^{(n)} \rangle)^2  \\
&=& K(e^{\sqrt{2}i\varepsilon W} -e^{-\sqrt{2}i\varepsilon W} \langle \phi_1^{(n)}, \phi_2^{(n)} \rangle^2).
\end{eqnarray*}
On the other hand, 
$$ d h_n(t) = d(e^{-\sqrt{2} i\varepsilon W(t)}) \langle \phi_1^{(n)}, \phi_2^{(n)} \rangle + e^{-\sqrt{2} i\varepsilon W(t)}\partial_t \langle \phi_1^{(n)}, \phi_2^{(n)}\rangle.$$
Therefore, letting $n \to \infty$ on both sides, we have 
$$ d (e^{-\sqrt{2}i\varepsilon W(t)} \langle \phi_1, \phi_2 \rangle) 
= d(e^{-\sqrt{2} i\varepsilon W(t)}) \langle \phi_1, \phi_2\rangle 
+K(1 -(e^{-\sqrt{2}i\varepsilon W} 
\langle \phi_1, \phi_2 \rangle)^2). $$
Since by It\^{o} formula 
$$ d(e^{-\sqrt{2} i\varepsilon W(t)}) = -\sqrt{2} i\varepsilon e^{-\sqrt{2}i\varepsilon W} dW -2\varepsilon^2 e^{-\sqrt{2}i\varepsilon W} dt,$$ 
and $h(t)= \langle \psi_1(t), \psi_2(t) \rangle = e^{-\sqrt{2}i\varepsilon W(t)} \langle \phi_1, \phi_2 \rangle$, we obtain the stochastic differential equation 
in the statement. 
\end{proof}

\section{Invariant probability measure}

In the previous section, we saw that $h(t) = \langle \psi_1(t), \psi_2(t) \rangle$ satisfies (\ref{eq:h}). 
We shall thus study the stochastic differential equation 
(\ref{eq:h}) in this section. 
Here we recall that our interest is in the case where  
$h(0)= \langle \psi_{1,0}, \psi_{2,0} \rangle$, therefore 
we need only consider initial conditions satisfying $\lvert h(0) \rvert \leq 1$ since  
$$
\lvert h(0) \rvert = \lvert \langle \psi_{1,0}, \psi_{2,0} \rangle \rvert \leq \lVert \psi_{1,0} \rVert \cdot \lVert \psi_{2,0} \rVert = 1.
$$
Moreover the $L^2$ conservation implies that $|h(t)| =| \langle \psi_1(t), \psi_2(t) \rangle| \le 1$.

We shall establish in the following proposition 
the existence and uniqueness of solution of (\ref{eq:h}) 
to identify the solution with $\langle \psi_1(t), \psi_2(t) \rangle$. 
We will see in fact that the following 
proposition describes a key property of solution of (\ref{eq:h}). 

\begin{proposition} \label{prop:comp_dans_B}
Let $\xi$ be $\mathcal{F}_0$-measurable  
such that $\prob(\lvert \xi \rvert \leq 1) = 1$. 
The stochastic differential equation (\ref{eq:h}) with 
$h(0)=\xi$ has a unique, $(\mathcal{F}_t)_{t\ge 0}$-adapted solution $h \in C([0,T], \C)$ a.s. for any $T>0$. 
Furthermore,
$$
\prob \left( \sup_{t \geq 0} \lvert h(t) \rvert \leq 1 \right) = 1.
$$ 
More precisely, 
\begin{itemize}
\item[(i)] 
If $\lvert \xi \rvert = 1$ a.s., then $\lvert h(t) \rvert = 1$ for all $t \geq 0$, a.s.
\item[(ii)] If $\lvert \xi \rvert < 1$ a.s., then $\lvert h(t) \rvert < 1$ for all $t \geq 0$, a.s. 
\end{itemize}
\end{proposition}

\begin{proof} The use of the classical existence and uniqueness theorem for stochastic differential equations with a locally Lipschitz coefficient yields the local existence of solutions of (\ref{eq:h}) for any initial condition (see for ex. \cite{kunze,oksendal10}). 
By It\^{o} formula, 
\begin{eqnarray*}
d \lvert h(t) \rvert^2 &=& 2 \mathrm{Re}[d h(t) \cdot \compconj{h(t)}] \\
&=& 2 \mathrm{Re}[(K(1 - h(t)^2) dt + \sqrt{2} i \varepsilon h(t) \circ dW) \cdot \compconj{h(t)}] \\
&=& 2K (\mathrm{Re}[\compconj{h(t)}] - \lvert h(t) \rvert^2 \mathrm{Re}[h(t)]) dt \\
&=& 2K \cdot \mathrm{Re}[h(t)] \cdot (1 - \lvert h(t) \rvert^2) dt.
\end{eqnarray*}
Thus
$$
\lvert h(t) \rvert^2 = (\lvert \xi \rvert^2 - 1) e^{-2K\int_0^t \mathrm{Re}[h(s)] ds} + 1, 
$$
which implies (i) and (ii). Accordingly, 
the initial datum $\xi$ under the condition $\prob(\lvert \xi \rvert \leq 1) = 1$ yields that the solution does not explode. \end{proof}

By this results, we may restrict the state space to $B$, $\partial B$ or $\compconj{B}$ and consider (\ref{eq:h}) on either one of these spaces. 

\subsection{Invariant probability measure on $\compconj{B}$}
Fix $\statespace \in \{ B, \partial B, \compconj{B}\}$. 
We denote by $(h^x(t))_{t \geq 0}$ the unique solution of (\ref{eq:h}) with initial condition $h(0) \sim \delta_x$ for $x \in \statespace$.
\vspace{3mm}

We recall that $(P_t^\statespace)_{t \geq 0}$ is the Markov semigroup associated with (\ref{eq:h}), namely, 
$$
P_t^{\statespace} f(x) = \mathbb{E}(f(h^x(t))), \quad t\ge 0
$$
for $f\in B_b (\statespace, \R)$.  

The Markov semigroup $(P_t^{\statespace})_{t \geq 0}$ also acts 
on the set of probability measures by duality. 
If $\mu$ is a probability measure on $\statespace$,  denote $\mu P_t^{\statespace}$ the action of $(P_t^{\statespace})_{t \geq 0}$ on $\mu$. This action is defined by
$$
\int_\statespace f d(\mu P_t^\statespace) = \int_\statespace (P_t^\statespace f) d\mu.
$$
With these definitions, we may write 
$$
\Law(h^{\xi}(t)) = \Law(h(0)) P_t^\statespace= \Law(\xi)P_t^{\statespace}.
$$
In particular $\Law(h^x(t)) = \delta_x P_t^\statespace$.

\begin{definition}
A probability measure $\mu$ on $\statespace$ is called invariant with respect to $(P_t^\statespace)_{t\ge 0}$ when
$$
\mu P_t^\statespace = \mu \text{ for all } t \geq 0.
$$
\end{definition}

\begin{proof}[Proof of Theorem \ref{thm:mesure_inv}]

See Theorem \ref{invariant-probability-measure-on-partial-b} for (i) and see Theorem \ref{invariant-probability-measure-on-b} for (ii). 
Here, we give a proof for (iii) admitting (i) and (ii). 
It is clear that the natural extension to $\compconj{B}$ of the unique invariant probability measure on $\bou{1}$ is an invariant probability measure on $\compconj{B}$. 

Let $\mu$ be an invariant probability measure on $\compconj{B}$. We define two new probability measures $\mu_{B}$ and $\mu_{\partial B}$ by setting
\begin{align*}
\mu_\itr{1} (F) &= \mu (F \cap \itr{1}) \\
\mu_\bou{1} (F) &= \mu (F \cap \bou{1})
\end{align*}
for every Borel set $F \in \mathcal{B}(\compconj{B})$. 
Since $\mu$ is invariant, the measure $\mu_\itr{1}$ and $\mu_\bou{1}$ are the natural extensions to $\compconj{B}$ of finite invariant measures on $\itr{1}$ and $\bou{1}$ respectively. Furthermore $\mu = \mu_\itr{1} + \mu_\bou{1}$. 

By (ii), there is no invariant probability measure on $\itr{1}$, there can be no finite invariant measures on $\itr{1}$ except the null measure. So it must be the case that $\mu_\itr{1} = 0$ and $\mu = \mu_\bou{1}$.
\end{proof}

\subsection{Invariant probability measure on $\partial B$} 








\begin{theorem}
\label{invariant-probability-measure-on-partial-b}
Consider the stochastic differential equation (\ref{eq:h}) 
on the restricted state space $\partial B$ and let $(P_t^{\partial B})_{t \geq 0}$ be the associated Markov semigroup. Then, 
\begin{enumerate}
\item $(P_t^{\partial B})_{t \geq 0}$ is strong Feller.
\item $(P_t^{\partial B})_{t \geq 0}$ admits a unique invariant measure $\mu_{K, \varepsilon}$.
\item $\mu_{K, \varepsilon}$ has a smooth density with respect to the Riemannian volume measure on $\partial B$.
\end{enumerate}
\end{theorem}

\begin{proof}
Note that $\partial B = S^1$ is a smooth one-dimensional manifold. 
Let $A_0, A_1 : \mathbb{C} \rightarrow \mathbb{C}$ be defined by
$$
A_0(h) = K(1 - h^2) \text{ and } A_1(h) = \sqrt{2} i \varepsilon h.
$$
Let $\tilde{A}_0$ and $\tilde{A}_1$ be the restrictions to $\partial B$ of $A_0$ and $A_1$ respectively. 
On $\partial B$ the stochastic differential equation (\ref{eq:h}) becomes
$$
d h(t) = \tilde{A}_0(h(t)) dt + \tilde{A}_1(h(t)) \circ dW
$$
with the definitions and notations of the chapter V of \cite{ikeda81}. 
It is clear that $\tilde{A}_1$ is never null on $\partial B$. Since $\partial B$ is a one-dimensional manifold, this implies that (\ref{eq:h}) is non-degenerate on $\partial B$. 
This implies (see \cite{kliemann87}) that the process is strong Feller. 
Since $\partial B$ is compact, the Krylov-Bogolyubov theorem implies that there exists a probability measure that is invariant with respect to $(P_t^{\partial B})_{t \geq 0}$. 
Because the diffusion is non-degenerate on $\partial B$, 
by the Stroock Varahdan support Theorem (page 44 of \cite{ak}), all points of $\partial B$ are accessible for all time $t\ge 0$.
Hence, corollary 2.7 of \cite{hairer16} (also \cite{dpz}) implies that $(P_t^{\partial B})_{t \geq 0}$ admits a unique invariant, thus ergodic  probability measure. Moreover, 
Theorem 4 of \cite{ichihara74} implies that this invariant probability measure has a smooth density with respect to the Riemannian measure on $\partial B$.
\end{proof}

We define the positive direction of $\partial B$ as, 
$\theta: -\pi \to \pi$ for the polar coordinate of $x \in \partial B \subset \C \cong \R^2$ with $x=(\mathrm{Re}(x),\mathrm{Im}(x))$: $\mathrm{Re}(x)
=\cos\theta$ and $\mathrm{Im}(x)= \sin \theta$.   
In case of an orientable Riemannian manifold, taking the positive orientation, the integral with respect to the Riemannian volume measure equals the integral with respect to the volume element (i.e. differential $1$-form $d\theta$ in our case), see Remark 5.2 in Section 5 of \cite{sakai} for details. 
We may thus calculate the explicit form of the density of $\mu_{K, \varepsilon}$ 
as follows.  

\begin{proposition} \label{prop:densite}
The unique invariant measure $\mu_{K, \varepsilon}$ 
of $(P_t^{\partial B})_{t \geq 0}$ is given by
$$
\mu_{K, \varepsilon}(A) =\frac{1}{\Gamma_{K, \varepsilon}} \int_{(\cos\theta, \sin\theta) \in A} 
\exp\left({\frac{2K}{\varepsilon^2}\cos \theta}\right)d\theta,
$$
where $A$ is a measurable set with respect to the Riemannian measure on $\partial B$, 
and $\Gamma_{K, \varepsilon}$ is the normalizing constant. i.e. $$\Gamma_{K, \varepsilon}=\int_{-\pi}^{\pi} 
\exp\left({\frac{2K}{\varepsilon^2}\cos \theta}  \right)d\theta.$$
\end{proposition}

\begin{proof}
Remark that the generator of Markov semigroup $(P_t^{\partial B})_{t \geq 0}$ 
is given by, in polar coordinates, 
$$
Lf(\theta) = (-2K \sin \theta) \frac{\partial f}{\partial \theta}(\theta) + \varepsilon^2 \frac{\partial^2 f}{\partial \theta^2}(\theta).
$$
for any $f \in C_b^2(\partial B, \R)$. It is thus enough to show that for any $f \in C_b^2(\partial B, \R)$,
$$
\int_{-\pi}^{\pi} Lf (\theta) 
\exp\left({\frac{2K}{\varepsilon^2}\cos \theta}\right) d\theta = 0.
$$
This is in fact immediate, by the integration by parts,
\begin{eqnarray*}
&& \int_{-\pi}^{\pi} Lf (\theta) \exp\left({\frac{2K}{\varepsilon^2}\cos \theta}\right) d\theta \\
&=& 
\int_{\pi}^{\pi} 
\Big[(-2K \sin \theta) \frac{\partial f}{\partial \theta}(\theta) + \varepsilon^2 \frac{\partial^2 f}{\partial \theta^2}(\theta) \Big]\exp\left({\frac{2K}{\varepsilon^2}\cos \theta}\right) d\theta \\
&=& \int_{-\pi}^{\pi} (-2K \sin\theta) \frac{\partial f}{\partial \theta}(\theta) \exp\left({\frac{2K}{\varepsilon^2}\cos \theta}\right) d\theta \\ 
&& \hspace{3mm}+ \Big[\varepsilon^2 \frac{\partial f}{\partial \theta}(\theta)\exp\left({\frac{2K}{\varepsilon^2} \cos \theta}\right)\Big]_{\theta=-\pi}^{\theta= \pi} \\
&& \hspace{3mm}- \int_{-\pi}^{\pi} \varepsilon^2 \frac{\partial f}{\partial \theta}(\theta) \frac{2K}{\varepsilon^2}(- \sin \theta) \exp \left( \frac{2K}{\varepsilon^2} \cos \theta \right) d\theta \\
&=& 0. 
\end{eqnarray*}
\end{proof}

We need the following proposition for later use. 

\begin{proposition} \label{prop:esp_re} 
For all $K, \varepsilon > 0$ we have 
$\E_{\mu_{K, \varepsilon}}[\mathrm{Re} \, \cdot] > 0$.
\end{proposition}

\begin{proof} Since we know explicitly the density 
from Proposition \ref{prop:densite},  
\begin{eqnarray*}
\E_{\mu_{K, \varepsilon}}[\mathrm{Re} \, \cdot ] 
&=& \frac{1}{\Gamma_{K, \varepsilon}} 
\int_{-\pi}^\pi ( \cos \theta) 
e^{\frac{2 K}{\varepsilon^2}\cos \theta} d\theta \\
&=& \frac{1}{\Gamma_{K, \varepsilon}} 
\left\{[(\sin \theta) e^{\frac{2 K}{\varepsilon^2} \cos \theta}]_{\theta = -\pi}^{\theta =\pi} 
- \int_{- \pi}^{\pi} (\sin \theta) 
\frac{-2K \sin \theta}{\varepsilon^2} e^{\frac{2K}{\varepsilon^2} \cos \theta} d\theta\right\} \\
&=& \frac{2K}{\varepsilon^2 \Gamma_{K, \varepsilon}} \int_{-\pi}^{\pi} (\sin \theta)^2 
e^{\frac{2K}{\varepsilon^2} \cos \theta} d\theta > 0.
\end{eqnarray*}
\end{proof}

\subsection{Large deviation principle for the invariant probability measure}

We will now deduce a large deviation principle using the explicit formula in Proposition \ref{prop:densite} for the invariant probability measure.

\begin{lemma}
For any Borel set $F \in \mathcal{B}(\partial B)$, let 
$$ m_{K, \varepsilon}(F)= \int_{(\cos\theta, \sin\theta) \in F} \exp\left(\frac{2K}{\varepsilon^2}\cos \theta\right) d\theta. 
$$
Then we have
\begin{eqnarray*}
&& 2 \sup_{(\cos\theta, \sin\theta) \in \mathring{F}} 
\cos \theta \le \liminf_{\varepsilon^2/K \rightarrow 0} \frac{\varepsilon^2}{K} \log m_{K, \varepsilon}(F) \\
&& \hspace{5mm}\le \limsup_{\varepsilon^2/K \rightarrow 0} \frac{\varepsilon^2}{K} \log m_{K, \varepsilon}(F) \le 2 \sup_{(\cos\theta, \sin\theta) \in \bar{F}} \cos \theta. 
\end{eqnarray*}
\end{lemma}
\begin{proof}
Let $F \in \mathcal{B}(\partial B)$ be an arbitrary fixed Borel set. We have 
\begin{eqnarray*}
\frac{\varepsilon^2}{K} \log m_{K, \varepsilon}(F) 
&=& \frac{\varepsilon^2}{K} \log \int_{(\cos\theta, \sin\theta) \in F} \exp\left(\frac{2K}{\varepsilon^2}\cos \theta\right) d\theta \\
&\leq& \frac{\varepsilon^2}{K} \log 
\int_{(\cos\theta, \sin\theta)\in F} \exp\left(\frac{2K}{\varepsilon^2}
\sup_{(\cos\theta, \sin\theta)\in \bar{F}} \cos \theta \right) d\theta\\
&\leq& \frac{\varepsilon^2}{K} \log \acc{ \exp\left(\frac{2K}{\varepsilon^2}\sup_{(\cos\theta, \sin\theta)\in \bar{F}} \cos \theta \right) \int_{(\cos\theta, \sin\theta) \in F}
1 d\theta }\\
&=& \acc{\frac{\varepsilon^2}{K} \log \int_{(\cos\theta, \sin\theta) \in F} 1 d\theta} + 2 \sup_{(\cos\theta, \sin\theta)\in \bar{F}} \cos \theta.
\end{eqnarray*}
Hence
$$
\limsup_{\varepsilon^2/K \rightarrow 0} \frac{\varepsilon^2}{K} \log m_{K, \varepsilon}(F) \le 2 \sup_{ (\cos\theta, \sin\theta)\in \bar{F}} \cos \theta.
$$
Let $\eta > 0$. There exists $\theta_0$ such that $(\cos\theta_0, \sin\theta_0)\in \mathring{F}$ with 
$\cos \theta_0 \geq \sup_{(\cos\theta, \sin\theta)\in \mathring{F}} \cos \theta - \eta$.  
Because $\cos$ is continuous, there exists an open set $U \subset \mathring{F}$ such that $(\cos\theta_0, \sin\theta_0) \in U$ and that $\cos \theta \geq \sup_{(\cos\theta, \sin\theta) \in \mathring{F}} \cos \theta - 2 \eta$ for all  
$\theta$ such that $(\cos\theta, \sin\theta)\in U$. 
We have
\begin{eqnarray*}
\frac{\varepsilon^2}{K} \log m_{K, \varepsilon}(F)
&=& \frac{\varepsilon^2}{K} \log \int_{(\cos\theta, \sin\theta) \in F} \exp\left(\frac{2K}{\varepsilon^2}\cos \theta\right) d\theta \\
&\geq& \frac{\varepsilon^2}{K} \log \int_{(\cos\theta, \sin\theta) \in U} \exp\left(\frac{2K}{\varepsilon^2}\cos \theta\right) d\theta \\
&\geq& \frac{\varepsilon^2}{K} \log \int_{(\cos\theta, \sin\theta) \in U} \exp\left(\frac{2K}{\varepsilon^2} (\sup_{(\cos\theta, \sin\theta) \in \mathring{F}} 
\cos \theta - 2 \eta)\right) d\theta \\
&=& \acc{\frac{\varepsilon^2}{K} \log \int_{(\cos\theta, \sin\theta) \in U} 1 d\theta} + 2\sup_{(\cos\theta, \sin\theta)\in \mathring{F}} \cos \theta - 4 \eta.
\end{eqnarray*}
Hence
$$
2 \sup_{(\cos\theta, \sin\theta) \in \mathring{F}} \cos \theta - 4 \eta \le \liminf_{\varepsilon^2/K \rightarrow 0} \frac{\varepsilon^2}{K} \log m_{K, \varepsilon}(F).
$$
Letting $\eta \rightarrow 0$ completes our proof.
\end{proof}

\begin{theorem}[Large deviation principle] \label{large_deviation_principle}
We have
$$
- \inf_{x \in \mathring{F}} I(x) \le \liminf_{\varepsilon^2/K \rightarrow 0} \frac{\varepsilon^2}{K} \log \mu_{K, \varepsilon}(F) \le \limsup_{\varepsilon^2/K \rightarrow 0} \frac{\varepsilon^2}{K} \log \mu_{K, \varepsilon}(F) \le - \inf_{x \in \compconj{F}} I(x)
$$
for all Borel sets $F \in \mathcal{B}(\partial B)$ where
$$
I(x) \defeq 2(1 - \Re[x]).
$$
\end{theorem}

\begin{proof}
Let $F \in \mathcal{B}(\partial B)$ be an arbitrary fixed Borel set. Notice
$$
\frac{\varepsilon^2}{K} \log \mu_{K, \varepsilon}(F) = \frac{\varepsilon^2}{K} \log m_{K, \varepsilon}(F) - \frac{\varepsilon^2}{K} \log m_{K, \varepsilon}(\partial B).
$$
Therefore,
\begin{align*}
\limsup_{\varepsilon^2/K \rightarrow 0} \frac{\varepsilon^2}{K} \log \mu_{K, \varepsilon}(F)
&\le \limsup_{\varepsilon^2/K \rightarrow 0} \frac{\varepsilon^2}{K} \log m_{K, \varepsilon}(F) - \liminf_{\varepsilon^2/K \rightarrow 0} \frac{\varepsilon^2}{K} \log m_{K, \varepsilon}(\partial B) \\
&\le 2 ( \sup_{(\cos\theta, \sin\theta)\in \bar{F}} \cos \theta - 1)
\end{align*}
where we used the previous lemma for the last inequality. 
Similarly,
\begin{align*}
\liminf_{\varepsilon^2/K \rightarrow 0} \frac{\varepsilon^2}{K} \log \mu_{K, \varepsilon}(F)
&\ge \liminf_{\varepsilon^2/K \rightarrow 0} \frac{\varepsilon^2}{K} \log m_{K, \varepsilon}(F) - \limsup_{\varepsilon^2/K \rightarrow 0} \frac{\varepsilon^2}{K} \log m_{K, \varepsilon}(\partial B) \\
&\ge 2 ( \sup_{(\cos\theta, \sin\theta)\in \mathring{F}} \cos \theta- 1 ).
\end{align*}

Hence
\begin{eqnarray*}
&& 2 \left(\sup_{(\cos\theta, \sin\theta) \in \mathring{F}} \cos \theta  -1\right) 
\le \liminf_{\varepsilon^2/K \rightarrow 0} \frac{\varepsilon^2}{K} \log \mu_{K, \varepsilon}(F) \\
&& \hspace{5mm} \le \limsup_{\varepsilon^2/K \rightarrow 0} \frac{\varepsilon^2}{K} \log \mu_{K, \varepsilon}(F) \le 2\left( \sup_{(\cos\theta, \sin\theta) \in \bar{F}} \cos \theta -1\right)
\end{eqnarray*}
which is equivalent to the desired inequalities.
\end{proof}

Finally, for the sake of completeness, we give a proof of Proposition 
\ref{prop:conv_eps_to_zero}.
\vspace{3mm}

\begin{proof}[Proof of Proposition \ref{prop:conv_eps_to_zero}]
It is enough to show that $\mu_{K, \varepsilon} \rightarrow \delta_1$ in distribution when $\frac{\varepsilon^2}{K} \rightarrow 0$.
Let $\varphi \in C_b(\partial B)$ be arbitrary fixed.
Let $\delta > 0$ be given. 
There exists $U \subset \partial B$ such that $1 \in U$ and that $\lvert \varphi(x) - \varphi(1) \rvert \leq \delta$ for all $x \in U$. It follows from the previous theorem that $\lim_{\varepsilon^2/K \rightarrow 0} \mu_{K, \varepsilon}(U^c) = 0$. In particular, there exists $\eta > 0$ such that $0 \le \mu_{K, \varepsilon}(U^c) \le \delta$ if $0 < \frac{\varepsilon^2}{K} < \eta$. And in that case
\begin{eqnarray*}
&&\lvert \E_{\mu_{K, \varepsilon}}(\varphi) - \varphi(1) \rvert \\
&\le& \lvert \E_{\mu_{K, \varepsilon}}(\varphi \textbf{1}_{U^c})\rvert + \lvert \E_{\mu_{K, \varepsilon}}(\varphi \textbf{1}_U) - \varphi(1) \mu_{K, \varepsilon}(U) \rvert + \lvert \varphi(1) \mu_{K, \varepsilon}(U) - \varphi(1) \rvert \\
&\le& \lVert \varphi \rVert_{L^\infty} \mu_{K, \varepsilon}(U^c) + \sup_{x \in U}\lvert \varphi(x) - \varphi(1) \rvert \mu_{K, \varepsilon}(U) + \lvert \varphi(1) \rvert \mu_{K, \varepsilon}(U^c) \\
&\le& \lVert \varphi \rVert_{L^\infty} \cdot \delta + \delta + \lVert \varphi \rVert_{L^\infty} \cdot \delta. 
\end{eqnarray*}
Thus $\E_{\mu_{K, \varepsilon}}(\varphi) \rightarrow \varphi(1)$ when $\frac{\varepsilon^2}{K} \rightarrow 0$. 
Because $\varphi \in C_b(\partial B)$ is arbitrary, this proves the desired convergence in distribution.
\end{proof}

\section{Convergence to the invariant probability measure}

\subsection{Convergence with initial condition in $B$}
In this subsection we investigate the asymptotic behavior 
of the law of solution with the initial data $x \in B$. 
To prove Theorem \ref{thm:conv_interieur}, it suffices 
to compare two solutions established in Proposition  \ref{prop:comp_dans_B}, i.e. the solution of (\ref{eq:h}) 
with initial data $x\in B$ and the stationary solution of (\ref{eq:h}) whose initial distribution is $\mu_{K, \varepsilon}$, unique invariant ergodic measure on $\partial B$. 

\begin{proposition} \label{prop:conv_hg}
Let $h^x(t)$ and $g(t)$ be two solutions of (\ref{eq:h}) such that $h^x(0) \sim \delta_x$ and 
$g(0) \sim \mu_{K, \varepsilon}$ for a fixed $x \in B$. 
Then, we have 
$$\lvert h^x(t) - g(t) \rvert \xrightarrow{t \rightarrow +\infty} 0$$ 
almost surely.
\end{proposition}

\begin{proof}
Set $D_t = h^x(t) - g(t)$. We see that $D_t$ satisfies 
$$ dD_t=-KD_t(h^x(t) + g(t)) dt + \sqrt{2} i \varepsilon D_t \circ dW.$$
Hence, we apply It\^{o} formula to obtain  
\begin{align*}
d\lvert D_t \rvert^2 
&= -2K \cdot \mathrm{Re}[h^x(t) + g(t)] 
\cdot \lvert D_t\rvert^2 dt.
\end{align*}
Similarly as in Proposition \ref{prop:comp_dans_B}, we have 
$$
\lvert D_t \rvert^2 = \lvert D_0 \rvert^2 e^{-2K \int_0^t \mathrm{Re}[h^x(s) + g(s)]ds},
$$
which leads to
$$
\lvert h^x(t) - g(t) \rvert^2 = \lvert h^x(0) - g(0) \rvert^2 e^{-2K\int_0^t \mathrm{Re}[h^x(s) + g(s)] ds}. 
$$
The ergodicity of $\mu_{K, \varepsilon}$ implies
$$
\lim_{t \rightarrow +\infty} \frac{1}{t} \int_0^t \mathrm{Re} [g(s)] ds = \E_{\mu_{K, \varepsilon}}[\mathrm{Re} \, \cdot] \text{ almost surely},
$$
and 
since $\E_{\mu_{K, \varepsilon}}[\mathrm{Re} \, \cdot] > 0$ 
by Proposition \ref{prop:esp_re}, we have
$$
\lim_{t \rightarrow +\infty} \int_0^t \mathrm{Re}[g(s)] ds  = \lim_{t \rightarrow +\infty} t \cdot \frac{1}{t} \int_0^t \mathrm{Re}[g(s)] ds = +\infty \text{ almost surely.}
$$
On the other hand, it follows from the proof of Proposition \ref{prop:comp_dans_B} that for $x \in B$,
$$
e^{-2K \int_0^t \mathrm{Re}[h^x(s)] ds} = \frac{1 - \lvert h^x(t) \rvert^2}{1 - \lvert x \rvert^2}
$$
so that 
\begin{align*}
\lvert h^x(t) - g(t) \rvert^2
&= \lvert h^x(0) -g(0) \rvert^2 e^{-2K \int_0^t 
\mathrm{Re}[h^x(s) + g(s)] ds } \\
&= \lvert h^x(0) -g(0) \rvert^2 
\frac{1 - \lvert h^x(t) \rvert^2}{1 - \lvert x \rvert^2} 
e^{-2K \int_0^t \mathrm{Re}[g(s)] ds } \\
&\leq \frac{\lvert h^x(0) -g(0) \rvert^2}{1 - \lvert x \rvert^2} e^{-2K \int_0^t \mathrm{Re}[g(s)] ds }
\end{align*}
almost surely. Finally, because $\lim_{t \rightarrow +\infty} \int_0^t \mathrm{Re}[g(s)] ds = +\infty$ almost surely, this implies that $$\lim_{t \rightarrow +\infty}\lvert g(t) - h^x(t) \rvert = 0$$ almost surely.
\end{proof}

 Here we remark that Theorem \ref{thm:conv_interieur} in the introduction follows from the fact that $g(t)$ in Proposition \ref{prop:conv_hg} satisfies $|g(t)|=1$ for all $t\ge 0$. Furthermore, Proposition \ref{prop:conv_hg} implies the convergence in law.
\vspace{3mm}

\begin{proof}[Proof of Corollary \ref{cor:conv_weak}]
Let $x \in B$ be arbitrary fixed.
Let $\varphi: \compconj{B} \rightarrow \R$ be a fixed bounded continuous function. Since $\compconj{B}$ is compact, the function $\varphi$ is uniformly continuous. Let $\omega : [0, +\infty] \rightarrow [0, +\infty]$ be the modulus of continuity of $\varphi$. Define the diffusions $(h^x(t))_{t \geq 0}$ and $(g(t))_{t \geq 0}$ as before 
with $h^x(0) \sim \delta_x$ 
and $g(0) \sim \mu_{K, \varepsilon}$ and compute
$$
\lvert \E_{\delta_x P_t^{B}}\left[ \varphi \right] - \E_{\mu_{K, \varepsilon}}\left[ \varphi \right]\rvert
= \lvert \E\left[ \varphi(h^x(t)) \right] - \E\left[ \varphi(g(t)) \right]\rvert \\
\leq \E\left[ \lvert \varphi(h^x(t)) - \varphi(g(t))\rvert \right] 
$$
Proposition \ref{prop:conv_hg} shows that
$$
\lvert h^x(t)-g(t) \rvert 
\xrightarrow{t \rightarrow +\infty} 0 \text{ almost surely},
$$
so
$$
0 \leq \lvert \varphi(h^x(t)) - \varphi(g(t))\rvert \leq \omega(\lvert h^x(t) - g(t)\rvert)
$$
implies
$$
\lvert \varphi(h^x(t)) - \varphi(g(t))\rvert \xrightarrow{t \rightarrow +\infty} 0 \text{ almost surely.}
$$
Furthermore, since $\lvert \varphi(h^x(t)) - \varphi(g(t)) 
\rvert \leq 2 \lVert \varphi \rVert_{L^\infty}$, 
the dominated convergence theorem yields
$$
\lim_{t \rightarrow +\infty} \E_{\delta_x P_t^{B}}[\varphi] = \E_{\mu_{K, \varepsilon}}[\varphi].
$$
Since $\varphi$ is arbitrary, 
this proves that $\delta_x P_t^{B}$ converges to $\mu_{K, \varepsilon}$ in distribution as $t \rightarrow +\infty$. 
\end{proof}

\vspace{0.4cm}

Taking the test function $\varphi(y)=\mathrm{Re}(y)$ for $y \in \bar{B}$ in the above proof,  we have 
for sufficiently large time $t>0$, 
$$
\E(\mathrm{Re} (h^x (t))) 
\ge
\frac{1}{2}\E(\mathrm{Re} (g(t)))
=
\frac{1}{2}\E_{\mu_{K, \varepsilon}}(\mathrm{Re} \, \cdot) >0,
$$
\vspace{3mm}
where the last positivity follows from 
Proposition \ref{prop:esp_re}, and 
this explains Remark \ref{rem:1} in the introduction. 
\vspace{3mm}

 Now we can prove that there exists no invariant probability measure on $B$.
\begin{theorem}
\label{invariant-probability-measure-on-b}
The Markov semigroup $(P_t^{B})_{t \geq 0}$ has no invariant probability measure.
\end{theorem}

\begin{proof}
Suppose that an invariant probability measure $\mu_0$ exists. Let $0<r<1$ be fixed. 
By the definition of the invariance of $\mu_0$, we have 
$$ \int_B \prob(|h^x(t)| \le r) \mu_0(dx) = \mu_0 (|x| \le r)$$
and LHS converges to zero as $t \to \infty$ by Theorem \ref{thm:conv_interieur}. 
Because of this, 
$$
1 = \mu_0(B) = \mu_0\left( \bigcup_{n \geq 1} \acc{\lvert x \rvert 
\leq 1 - \frac{1}{n}} \right) 
\leq \sum_{n \geq 1} \mu_0
\left(\acc{\lvert x \rvert \leq 1 - \frac{1}{n}}\right) = 0.
$$
This contradiction concludes our proof.
\end{proof}





\subsection{Convergence with initial condition in $\partial B$}

In this subsection, we restrict the state space to $\partial B$ and we investigate the asymptotic behavior of $\delta_x P_t^{\partial B}$ for fixed $x \in \partial B$. Since the state space $\partial B$ is compact, we have a uniform ``minorisation'' condition (\cite{m}) reminiscent of Doeblin's condition, which implies an exponential convergence to the invariant measure. For the sake of completeness we give a proof of the convergence, following \cite{m, msh}.

\begin{lemma} \label{lem:minorization}
There exists $\eta > 0$ and a probability measure $\nu$ such that
$$
\delta_x P_1^{\partial B}(F) \ge \eta \nu(F)
$$
for all measurable sets $F$ on $\partial B$ and all $x \in \partial B$.
\end{lemma}




Recall the two following properties. 

\begin{proposition}
We have $\text{supp} \, \delta_x P_t^{\partial B} = \partial B$ for every $x \in \partial B$ and $t >0$ so in particular
$$
\delta_x P_t^{\partial B}(U) > 0
$$
for every $x \in \partial B$, $t>0$ and every open set $U$.
\end{proposition}

\begin{proof}
The diffusion process corresponding to (\ref{eq:h}) on $\partial B$ is a non-degenerate diffusion with smooth coefficients on a smooth connected manifold. 
Thus the desired result follows immediately from the Stroock-Varadhan support theorem (see page 44 of \cite{ak}). 
\end{proof}

\begin{proposition}
The semigroup $P_t^{\partial B}$ admits a continuous density with respect the Riemannian measure $\lambda$ on $\partial B$, namely for all $t > 0$
$$
\delta_x P_t^{\partial B} (F) = \int_F p_t(x, y) d\lambda(y)
$$
for all measurable sets $F$ on $\partial B$. 
Furthermore $(x, y) \mapsto p_t(x, y)$ is continuous for fixed $t > 0$.
\end{proposition}

\begin{proof}
This follows again from the non degeneracy of our diffusion with smooth coefficients on a smooth connected manifold.
\end{proof}

We are now in a position to prove Lemma \ref{lem:minorization}, 
our proof follows the paper \cite{msh}. 

\begin{proof}[Proof of lemma \ref{lem:minorization}]
Since $\int_{\partial B} p_{\frac12}(1, z) d\lambda(z) = 1$ there exists $z^* \in \partial B$ such that
$$
p_{\frac12}(1, z^*) > 0.
$$
Because of the continuity of $p_{\frac12}$ there exist $\varepsilon > 0$ such that if $y \in B_{\varepsilon}(1)$ and $z \in B_{\varepsilon}(z^*)$ we have
$$
p_{\frac12}(y, z) > \frac{1}{2} p_{\frac12}(1, z^*).
$$
Therefore for any $y \in B_{\varepsilon}(1)$ we have
$$
\delta_y P_{\frac12}^{\partial B}(F) = \int_{F} p_{\frac12}(y, z) d\lambda(z) \geq \int_{F \cap B_{\varepsilon}(z^*)} p_{\frac12}(y, z) d\lambda(z) \geq \frac{1}{2} p_{\frac12}(1, z^*) \lambda(F \cap B_{\varepsilon}(z^*)).
$$
Furthermore, the continuity of $p_{\frac12}$ implies the continuity of $x \mapsto \delta_x P_{\frac12}^{\partial B}(F)$ for fixed measurable sets $F$ by the dominated convergence theorem. 

Since $\partial B$ is compact, this implies that $x \mapsto \delta_x P_{\frac12}^{\partial B}(F)$ reaches its infimum on $\partial B$ and hence
$$
\inf_{x \in \partial B} \delta_x P_{\frac12}^{\partial B}(B_\varepsilon(1)) > 0.
$$
Then $\delta_x P_1 (F)$ is estimated from below as follows. 
\begin{align*}
\delta_x P_1(F) &= \int_{\partial B} \delta_y P_{\frac12}^{\partial B}(F) p_{\frac12}(x, y) d\lambda(y) \\
&\ge \int_{B_{\varepsilon}(1)} \delta_y P_{\frac12}^{\partial B}(F) p_{\frac12}(x, y) d\lambda(y) \\
&\ge \int_{B_{\varepsilon}(1)} \frac{1}{2} p_{\frac12}(1, z^*) \lambda(F \cap B_{\varepsilon}(z^*)) p_{\frac12}(x, y) d\lambda(y) \\
&= \frac{1}{2} p_{\frac12}(1, z^*) \lambda(F \cap B_{\varepsilon}(z^*)) \delta_x P_{\frac12}^{\partial B}(B_{\varepsilon}(1)) \\
&\ge \frac{1}{2} p_{\frac12}(1, z^*) \lambda(F \cap B_{\varepsilon}(z^*)) (\inf_{x \in \partial B} \delta_x P_{\frac12}^{\partial B}(B_{\varepsilon}(1))) \\
&= \eta \nu(F),
\end{align*}
where $\eta = \frac{1}{2} p_1(1, z^*) (\inf_{x \in \partial B} \delta_x P_{\frac12}^{\partial B}(B_{\varepsilon}(1))) \lambda(B_{\varepsilon}(z^*))$ and $$\nu(F) = \frac{\lambda(F \cap B_{\varepsilon}(z^*))}{\lambda(B_{\varepsilon}(z^*))}.$$
\end{proof}

Once the minorization property is established, we directly have a contraction property of the law of solution. 

Let us recall that given two probability measures $\mu_1$ and $\mu_2$ on $\partial B$, their total variation distance is 
\begin{eqnarray*}
\|\mu_1- \mu_2\|_{\text{TV}} &=& 
\sup \{|(\mu_1-\mu_2)(\Gamma)|:~\Gamma \in \mathcal{B}(\partial B)\} \\
&=& \sup \left\{ \left|\int_{\partial B} \varphi d\mu_1 - \int_{\partial B} \varphi d \mu_2 \right|:~ \varphi \in B_{b}(\partial B, \R), ~ \|\varphi\|_{L^{\infty}(\partial B)} \le 1 \right\}.    
\end{eqnarray*}
\vspace{3mm}

\begin{lemma} \label{lem:contraction}
There exists $\alpha > 0$ such that for all $x, y \in \partial B$ we have
$$
\|\delta_x P_1^{\partial B} - \delta_y P_1^{\partial B}\|_{\text{TV}} 
\leq 1 - \alpha.
$$
\end{lemma}

\begin{proof}
By the definition of the total variance distance,  
$$\lVert \delta_x P_1^{\partial B} - \delta_y P_1^{\partial B} \rVert_{\text{TV}} = \sup_{F \in \mathcal{B}(\partial B)} \lvert \delta_x P_1^{\partial B} (F) - \delta_y P_1^{\partial B}(F)\rvert.$$
So to obtain the desired result it is enough to show that for all $x, y \in \partial B$ and all $F \in \mathcal{B}(\partial B)$ we have
$$
\lvert \delta_x P_1^{\partial B} (F) - \delta_y P_1^{\partial B}(F) \rvert \leq 1 - \alpha
$$
for a certain $\alpha > 0$. 

Take the probability measure $\nu$ and $\eta>0$ in Lemma \ref{lem:minorization}.  
When $\nu(F) \geq \frac{1}{4}$, we have 
$\frac{1}{4} \eta \leq \delta_x P_1^{\partial B}(F) \leq 1$ for all $x \in \partial B$ and hence
$$
\lvert \delta_x P_1^{\partial B}(F) - \delta_y P_1^{\partial B}(F) \rvert \leq 1 - \frac{1}{4} \eta,
$$
for all $x, y \in \partial B$. 
Consider next the case $\nu(F) < \frac{1}{4}$. 
If we assume that there exists $x \in \partial B$ 
such that $\delta_x P_1^{\partial B} (F) > 1 - \frac{1}{4} \eta$, 
this would imply
\begin{align*}
1 &= \delta_x P_1^{\partial B}(F) + \delta_x P_1^{\partial B}(F^c) \\
&> 1 - \frac{1}{4} \eta + \eta \nu(F^c) \\
&= 1 - \frac{1}{4} \eta + \eta (1 - \nu(F)) \\
&= 1 + \eta (1 - \frac{1}{4} - \nu(F)) > 1,
\end{align*}
which is a contradiction, thus $\delta_x P_1^{\partial B} (F) \leq 1 - \frac{1}{4} \eta$ for all $x \in \partial B$. And therefore
$$
\lvert \delta_x P_1^{\partial B}(F) - \delta_y P_1^{\partial B}(F) \rvert \leq 1 - \frac{1}{4} \eta
$$
for all $x, y \in \partial B$.
\end{proof}

We then give a proof of Proposition \ref{prop:conv_bord}. 
\vspace{3mm}

\begin{proof}[Proof of Proposition \ref{prop:conv_bord}]
First, we show that Lemma \ref{lem:contraction} implies 
that for any $t\ge 1$ and $x,y \in \partial B$, 
$$
\|\delta_x P_t^{\partial B} - \delta_y P_t^{\partial B}\|_{\mathrm{TV}} \le 1-\alpha.
$$
Indeed, by the Markov property, for $\phi \in C_b(\partial B)$ with $\|\phi\|_{L^{\infty}(\partial B)} \le 1$ and $t \ge 1$, 
\begin{eqnarray*}
|P_t^{\partial B} \phi(x) - P_t^{\partial B} \phi(y)|
&=&|\E(\phi(h^x(t)))-\E(\phi(h^y(t)))| \\
&=&|\E(P_1^{\partial B}\phi(h^x(t-1))-\E(P_1^{\partial B}\phi(h^y(t-1))|\\
&=&\left|\int_{\partial B \times \partial B} 
[P_1^{\partial B} \phi(\tilde{x}) -P_1^{\partial B} \phi (\tilde{y})] 
\prob_{t-1}^{x,y}(d \tilde{x}, d\tilde{y})\right|,  
\end{eqnarray*}
where in the last equality we used  
the trivial product coupling
$$
\prob_{s}^{x,y}(A \times B) 
=\Big( \delta_x P^{\partial B}_s(A)\Big)\Big( \delta_y P_s^{\partial B}(B)\Big)
$$
with measurable sets $A,B$ in $\partial B$ and $s\ge 0$.
Since we know that $\delta_x P^{\partial B}_{t-1}(\partial B)= \delta_y P^{\partial B}_{t-1}(\partial B)=1$, by Lemma \ref{lem:contraction}, we obtain for $t \ge 1$, 
\begin{eqnarray*}
|P_t^{\partial B} \phi(x) - P_t^{\partial B} \phi(y)|
 \le (1-\alpha) \prob_{t-1}^{x,y}(\partial B \times \partial B)
=1-\alpha,
\end{eqnarray*}
which concludes the inequality above. 
Next, it can be seen that for any probability measures $\mu_1$ and $\mu_2$, 
and for any coupling $M$ on $\partial B \times \partial B$ of $\mu_1$ and $\mu_2$, 
we have for $t\ge 1$,
\begin{eqnarray*}
\|\mu_1 P_t^{\partial B} - \mu_2 P_t^{\partial B}\|_{\mathrm{TV}} 
&\le& \sup_{\|\phi\|_{L^{\infty}} \le 1} 
\int \left|P_t^{\partial B} \phi(x) - P_t^{\partial B} \phi (y)\right| M(dx,dy)\\ 
&\le& (1-\alpha) \left(1- M\{(x,x), x\in \partial B\}\right). 
\end{eqnarray*}
Note that 
$$\|\mu_1 - \mu_2 \|_{\mathrm{TV}}=
\inf \left\{1- M\{(x,x), x\in \partial B\}|~ M:\mbox{coupling of }~\mu_1~\mbox{and}
~\mu_2 \right\},$$
thus, for all $t\ge 1$,
$$  
\|\mu_1 P_t - \mu_2 P_t\|_{\mathrm{TV}} \le 
(1-\alpha) \|\mu_1-\mu_2\|_{\mathrm{TV}}.
$$
Finally, taking $\mu_1=\delta_x$, $\mu_2=\mu_{K, \varepsilon}$, we repeat recurrently this inequality to get 
$$
\|\delta_x P_t - \mu_{K, \varepsilon}\|_{\mathrm{TV}} \le 
(1-\alpha)^{[t]}\|\delta_x-\mu_{K, \varepsilon}\|_{\mathrm{TV}},
$$
for any $t\ge 1$.
\end{proof}
\vspace{3mm}

\begin{remark} Let $\mathcal{U}_1 \subset \partial B$ be a small neighborhood of the point $1$. Proposition \ref{prop:conv_bord} implies that $\lim_{t\to +\infty} \prob(h^x(t) \in \mathcal{U}_1) = \mu_{K,\varepsilon}(\mathcal{U}_1)$ for all $x \in \partial B$. On the other hand, we see that 
$\mu_{K,\varepsilon}(\mathcal{U}_1) >0$ from Proposition \ref{prop:densite}. 
Hence, it follows from  Proposition 3.4.5 of \cite{dpz} that for an arbitrary sequence $\{t_n\} \to +\infty,$ 
$$
\prob(h^x(t_n) \in \mathcal{U}_1, ~\mbox{for infinitely many}~ n \in \N)=1.
$$ 
\end{remark}
\vspace{3mm}

{\bf Acknowlegment}
This work was partly supported by JSPS KAKENHI Grant Numbers JP19KK0066,
JP20K03669.
The authors are grateful to have useful discussions with D. Kim. 
The authors would also like to express their  gratitude to K.~Funano, Y. Hariya, A. de Bouard and D. Chafa\"{i} for their encouragements and discussions.


\begin{thebibliography}{9}
\bibitem{antonelli17}
P. Antonelli and P. Marcati,
\textit{A model of synchronization over quantum networks}, J. Phys. A. \textbf{50}, no. 31 (2017)

\bibitem{ak} L. Arnold and W. Kliemann, 
\textit{On unique ergodicity for degenerate diffusions},  
Srochasrics. {\bf 21} 41--61 (1987)

\bibitem{dbf} A. de Bouard and R. Fukuizumi, 
\textit{Representation formula for stochastic Schr\"odinger 
evolution equations and applications}, Nonlinearity {\bf 25} 
2993--3022 (2012)

\bibitem{c} T. Cazenave, 
\textit{Semilinear Schr\"{o}dinger Equations}, Courant Lecture Notes in Mathematics {\bf 10} New York University, Courant Institute of Mathematical Sciences, AMS, 2003.

\bibitem{choi16}
S.-H. Choi, J. Cho and S.-Y. Ha,
\textit{Practical quantum synchronization for the Schrödinger-Lohe system},
J. Phys. A: Math. Theor. \textbf{49} (2016) 205203 (17pp)

\bibitem{choi14}
S.-H. Choi and S.-Y. Ha,
\textit{Quantum synchronization of the Schrödinger-Lohe model},
J. Phys. A: Math. Theor. \textbf{47} (2014) 355104 (16pp)

\bibitem{dpz} G. Da Prato and J. Zabczyk, \textit{Ergodicity for infinite dimensional systems,} London Mathematical Society Lecture Note Series 229, Cambridge University press (1996).

\bibitem{fv} M. I. Freidlin and A. D. Wentzell, 
\textit{Random perturbations of dynamical systems}, 
second edition. Springer. (1998) 

\bibitem{gp} D.S. Goldobin and A. Pikovsky, 
\textit{Synchronization and desynchronization of self-sustained 
oscillators by common noise},
Phys. Rev. E. {\bf 71}
045201(R) (2005) 


\bibitem{hairer16}
M. Hairer,
\textit{Convergence of Markov Processes} (2016) \\
\url{http://hairer.org/notes/convergence.pdf}

\bibitem{huh18}
H. Huh, S.-Y. Ha and D. Kim,
\textit{Asymptotic behavior and stability for the Schrödinger-Lohe model},
J. Math. Phys. \textbf{59}, 102701 (2018)

\bibitem{ichihara74}
K. Ichihara and H. Kunita,
\textit{A classification of the second order degenerate elliptic operators and its probabilistic characterization}, Z. Wahrscheinlichkeitstheorie verw. Gebiete \textbf{30}, 235-254 (1974)

\bibitem{ikeda81} N. Ikeda and S. Watanabe,
\textit{Stochastic Differential Equations and Diffusion Processes}, North-Holland Mathematical Library {\bf 24} (1981)

\bibitem{kim20}
D. Kim and J. Kim,
\textit{Stochastic Lohe Matrix Model on the Lie Group and Mean-Field Limit}, J. Stat. Phys. \textbf{178}, 
1467-1514 (2020)

\bibitem{kliemann87}
W. Kliemann,
\textit{Recurrence and Invariant Measures for Degenerate Diffusions}, Ann. Probab. {\bf 15} Number 2 690-707 (1987)

\bibitem{kunze} M. Kunze, 
\textit{Stochastic differential equations},
Lecture notes, summer term 2012, Universit\"{a}t Ulm. 

\bibitem{k} Y. Kuramoto, Chemical Oscillations, Waves, and Turbulence, Dover Publications. Inc. Mineola, NewYork 2003  (originally published in Springer-Verlag 1984) 

\bibitem{lohe10}
M. Lohe,
\textit{Quantum synchronization over quantum networks}, J. Phys. A. {\bf 43} Number 46 (2010)

\bibitem{lohe09}
M. Lohe,
\textit{Non-abelian Kuramoto model and synchronization},
J. Phys. A. {\bf 42} Number 39 (2009)

\bibitem{m} J. C. Mattingly, \textit{Exponential convergence for the stochastically forced Navier-Stokes equations and other partially dissipative dynamics},
Commun. Math. Phys. {\bf 230} 421--461 (2002)

\bibitem{msh} J. C. Mattingly, A. M. Stuart and D. J. Higham, 
\textit{Ergodicity for SDEs and approximations: locally Lipschitz vector fields and degenerate noise}, 
Stochastic Processes and their Applications. {\bf 101} 185--232 (2002)


\bibitem{oksendal10}
B. {\O}ksendal,
\textit{Stochastic differential equations: an introduction with applications},
6th edition (2010)

\bibitem{sakai} T. Sakai, 
\textit{Riemannian Geometry}, 
Translation of Mathematical Monographs {\bf 149} 
A.M.S. (1992)

\bibitem{tt} J. Teramae and D. Tanaka, Robustness of the noise-induced phase synchronization in a general class of limit cycle oscillators,
Phys. Rev. Lett. {\bf 93} 204103:1-4 (2004)

\end{thebibliography}
\end{document}